\documentclass[12pt,reqno]{amsart}
\usepackage{amsthm}
\usepackage{fullpage,url}
\usepackage{pstricks,pst-plot,pst-node}
 \usepackage{amscd}   % for commutative diagrams
\usepackage[all,graph,poly]{xy} % for complicated commutative diagrams
\usepackage{amssymb,amsmath}
\usepackage{multicol}
\usepackage{setspace}
\DeclareFontEncoding{OT2}{}{} % to enable usage of cyrillic fonts

\newcommand{\abs}[1]{\lvert#1\rvert}

\newcommand{\br}{\bar{\rho}}
\newcommand{\cA}{\mathcal{A}}

\newcommand{\cE}{\mathcal{E}}

\newcommand{\cH}{\mathcal{H}}
\newcommand{\cL}{\mathcal{L}}
\newcommand{\cM}{\mathcal{M}}

\newcommand{\cP}{\mathcal{P}}

\newcommand{\CC}{\mathbb{C}}

\newcommand{\DD}{\mathbb{D}}
\newcommand{\FF}{\mathbb{F}}

\newcommand{\ip}[2]{\langle #1, #2 \rangle}

\newcommand{\nr}{\abs{\br}^2}

\newcommand{\op}[1]{\operatorname{#1}}
\newcommand{\PP}{\mathbb{P}}
\newcommand{\PM}[4]{ \begin{pmatrix} #1 & #2 \\ #3 & #4 \end{pmatrix} }
\newcommand{\PV}[2]{ \begin{pmatrix} #1  \\ #2  \end{pmatrix} }
\newcommand{\QQ}{\mathbb{Q}}

\newcommand{\RR}{\mathbb{R}}

\newcommand{\SV}[2]{\bigl( \begin{smallmatrix} #1  \\ #2  \end{smallmatrix} \bigr) }

\newcommand{\ZZ}{\mathbb{Z}}
\DeclareMathOperator{\Aut}{Aut}

\DeclareMathOperator{\Conv}{Conv}
\DeclareMathOperator{\Cox}{Cox}
\DeclareMathOperator{\Geod}{Geod}
\DeclareMathOperator{\md}{md}
\DeclareMathOperator{\pr}{pr}

\providecommand{\abs}[1]{\lvert#1\rvert}
\newcommand{\arhh}[1]{\ar@{->>}[#1]}
\newcommand{\e}[1]{\ar@{-}[#1]}
\newcommand{\ed}[1]{\ar@{--}[#1]}
\newcommand{\ee}[1]{\ar@{=}[#1]}
\newcommand{\er}[1]{\ar@{-}[#1] |-{\SelectTips{cm}{}\object@{<}} |{\SelectTips{eu}{}\object@{}}}
\newcommand{\el}[1]{\ar@{-}[#1] |-{\SelectTips{cm}{}\object@{>}} |{\SelectTips{eu}{}\object@{}}}
\newcommand{\edr}[1]{\ar@{--}[#1] |-{\SelectTips{cm}{}\object@{<}} |{\SelectTips{eu}{}\object@{}}}
\newcommand{\edl}[1]{\ar@{--}[#1] |-{\SelectTips{cm}{}\object@{>}} |{\SelectTips{eu}{}\object@{}}}

 \newcommand{\lul}[1]{\ar@{}[l]_<<{#1}}
\newcommand{\rrul}[1]{\ar@{}[r]^<<<<{#1}}
\newcommand{\rul}[1]{\ar@{}[r]^<<{#1}}
\newcommand{\ldl}[1]{\ar@{}[l]^<<{#1}}
\newcommand{\rdl}[1]{\ar@{}[r]_<<{#1}}
\newcommand{\dl}[1]{\ar@{}[d]_<<{#1}}
\newcommand{\dll}[1]{\ar@{}[dd]_{#1}}
\swapnumbers
\newtheorem*{thm}{Theorem}
\newtheorem{theorem}{Theorem}[section]
\newtheorem{lemma}[theorem]{Lemma}

\theoremstyle{definition}
\newtheorem{definition}[theorem]{Definition}

\newtheorem{conjecture}[theorem]{Conjecture}

\newtheorem{topic}[theorem]{}

\theoremstyle{remark}
\newtheorem{remark}[theorem]{Remark}
\newtheoremstyle{head}% name
{}% Space above
{}% Space below
{\bfseries}% Body font
{}% Indent amount (empty = no indent, \parindent = para indent)
{}% Thm head font
{}% Punctuation after thm head
{.5em}% Space after thm head: " " = normal interword space;
{}% Thm head spec (can be left empty, meaning `normal')

\theoremstyle{head}

\begin{document}
%
%*******************************************************************************************************
%
%
%*******************************************************************************************************
%
\title[el2]{The complex Lorentzian Leech lattice and the bimonster (II)}
\author{Tathagata Basak}
\address{Department of Mathematics\\Iowa State University, \\Ames, IA 50011}
\email{tathagat@iastate.edu}
\urladdr{http://orion.math.iastate.edu/tathagat}
%
%(11H : geometry of numbers) 11H56: Automorphism groups of Lattices
%
%(20D : Abstract finite groups) 20D08 Simple groups: sporadic groups
%
%20F : Special aspects of finite or infinite groups
%20F05 Generators, relations, and presentations, 20F55: Reflection and Coxeter groups, 20F36: Braid groups; Artin groups
%
%(51F : Metric geometry) 51F15: Reflection groups, reflection geometries
%(51M : Real and complex geometry) 51M10 : Hyperbolic and elliptic geometries (general) and generalizations
%
%
\date{April 4, 2012}
\begin{abstract} Let $D$ be the incidence graph of the projective plane over $\FF_3$.
The Artin group of the graph $D$ maps onto the bimonster and a complex hyperbolic
reflection group $\Gamma$ acting on $13$ dimensional complex hyperbolic space $Y$.
The generators of the Artin group are mapped to elements of order $2$ (resp. $3$)
in the bimonster (resp. $\Gamma$). Let $Y^{\circ} \subseteq Y$ be the complement of
the union of the mirrors of $\Gamma$. Daniel Allcock has
conjectured that the orbifold fundamental group of $Y^{\circ}/\Gamma$ surjects
onto bimonster. 
\par
In this article we study the reflection group $\Gamma$. 
Our main result shows that there is homomorphism from the
Artin group of $D$ to the orbifold fundamental group of $Y^{\circ}/\Gamma$,
obtained by sending the Artin generators to the generators of monodromy around the
mirrors of the generating reflections in $\Gamma$. 
This answers a question in Allcock's article ``A monstrous proposal'' and takes
a step towards the proof of Allcock's conjecture. The finite group
$ \op{PGL}(3, \FF_3)  \subseteq \Aut(D)$ acts on $Y$ and fixes a complex hyperbolic line pointwise.
 We show that the restriction of $\Gamma$-invariant meromorphic automorphic forms on $Y$
to the complex hyperbolic line fixed by $\op{PGL}(3, \FF_3)$ gives meromorphic
modular forms of level $13$.
\end{abstract}
\maketitle
%
%{\bf {\large The complex Lorentzian Leech lattice and the bimonster (II)}}
%\par
%\vspace{.5cm}
%Author : {\large Tathagata Basak}
%\par
%\vspace{.5cm}
%{\small Address : Department of Mathematics\\Iowa State University, \\Ames, IA 50011}
%\par
%{\small email : tathagat@iastate.edu}\\
%
%
%*******************************************************************************************************
%
%
%
\section{Introduction}
This article is a continuation of \cite{TB:EL}. Here we continue our study
of the reflection group of the complex Lorentzian Leech lattice. Before describing
our results (see \ref{t-summary}), we briefly recall the context, which makes the study of this
particular reflection group interesting.
\begin{topic}{\bf Some notation and background: }
Let $\DD$ be a graph with vertex set $\lbrace x_1, \dotsb, x_k \rbrace$.
Let $\cA(\DD)$ be the group generated by $k$ generators, also denoted by
$x_1, \dotsb, x_k$, and the relations 
\begin{align*}
x_i x_j &= x_j x_i \text{\; if \;} \lbrace x_i, x_j \rbrace \text{\; is not an edge of \;} \DD, \\
x_i x_j x_i &= x_j x_i x_j \text{\; if \;} \lbrace x_i, x_j \rbrace \text{\; is an edge of \;} \DD, 
\end{align*}
for all $i$ and $j$. The group $\cA(\DD)$ is called the {\it Artin group} of the graph $\DD$.
Let $\Cox(\DD,n)$ be the quotient of $\cA(\DD)$, obtained by imposing the relations $x_i^n = 1$ for
all $x_i \in \DD$. In this notation, the Coxeter group of the graph $\DD$ is $\Cox(\DD,2)$.
\par
The wreath product of the monster simple group with $\ZZ/2\ZZ$ is known as the
bimonster. Conway and Norton conjectured a simple presentation of the bimonster
that later became the Ivanov-Norton theorem. 
\begin{thm}[\cite{AAI:Geom}, \cite{AAI:Y}, \cite{SN:CM}]
Let $M_{6 6 6}$ the graph, shaped like an `` $Y$'', with $16$ vertices (six
in each hand including the central vertex). Label the successive vertices in the
$i$--th hand by $a, b_i, c_i, d_i, e_i, f_i$ with $a$ being the central vertex.
Then $\Cox(M_{6 6 6}, 2)$ maps onto the bimonster and the kernel is
generated by the single relation $(a b_1 c_1 a b_2 c_2 a b_3 c_3)^{10} = 1$.
\end{thm}
Using the Ivanov-Norton theorem, Conway et.al. obtained a second presentation
of the bimonster. 
Let $D$ be the incidence graph of the projective plane over the finite field $\FF_3$.
The graph $M_{6 6 6}$ is a maximal sub-tree of $D$, so there is a natural map from
$\Cox(M_{6 6 6}, 2)$ to $\Cox(D,2)$.
\begin{thm}[\cite{CCS:26}, also see \cite{CNS:B}, \cite{CP:H}]
The surjection from $\Cox(M_{6 6 6}, 2)$ to the bimonster, extend to a surjection from
$\Cox(D, 2)$ to the bimonster. The kernel is generated by some 
explicitly described simple relations called ``deflating the $12$--gons''.
\end{thm}
We need a bit of notation to introduce the reflection group. Let $\omega = e^{2 \pi i/3}$
and $\cE = \ZZ[\omega]$. Define the $\cE$--lattice $L$ to be the direct sum of the complex
Leech lattice and a hyperbolic cell (See 2.4-2.5 of \cite{TB:EL}).
Equivalently, we may define $L$ to be  the unique Hermitian lattice defined over $\cE$ such
that $(2 +  \omega) L' = L$ (where $L'$ denotes the dual lattice of $L$). 
Let $Y = \PP_+(L^{\CC})$ be the set of complex lines of positive norm in the complex
vector space $L^{\CC} = L \otimes_{\cE} \CC$. The space $Y$ is isomorphic
to $13$ dimensional complex hyperbolic space. The projective automorphism
group $\PP \Aut(L)$ acts faithfully on $Y$. We write $\Gamma = \PP \Aut(L)$.
\par
Daniel Allcock showed in \cite{DA:Leech} that the reflection group of $L$, denoted by $R(L)$,
has finite index in $\Aut(L)$. So $R(L)$ is an arithmetic subgroup of $U(1,13)$. This example
provides the largest dimension in which an arithmetic hyperbolic reflection group is known.
Allcock also observed that there is a map $\phi: \Cox(M_{6 6 6}, 3) \to R(L)$ sending the
generators to complex reflections of order three. While trying to better understand the
reflection group of $L$, we made the following observations:
\begin{theorem}[\cite{TB:EL}] 
(a) The map $\phi: \Cox(M_{6 6 6}, 3) \to R(L)$ is onto and it
extends to a surjection $ \phi : \Cox(D, 3) \to R(L)$. In other words, there are $26$ complex
reflections of order $3$ (called {\it simple reflections}) in the reflection group of $L$ which braid
or commute according to the diagram $D$ and these reflections generate
$R(L)$. (The proof of this fact, given in \cite{TB:EL} used a computer
calculation. Since then, Allcock has found a computer free proof; see \cite{DA:Y555}).
\newline
(b) One has $\Aut(L) = R(L)$. The ``deflation relation'' holds in $R(L)$. 
\newline
(c) (see prop. 6.1 in \cite{TB:EL}) Let $D^{\bot}$ be the set of $26$ mirrors, fixed by the $26$
simple reflections. The group $\Gamma$ possesses a subgroup $Q \simeq 2.\op{PGL}(3, \FF_3)$ (called the group
of diagram automorphisms) which permutes the mirrors $D^{\bot}$ in the same way as $2.\op{PGL}(3, \FF_3)$ permutes
the points  and lines of $\PP^2(\FF_3)$. The action of $Q$ fixes a unique point $\br$ in
the complex hyperbolic space $Y$. The set $D^{\bot}$ consists of precisely the mirrors that
are closest to $\br$. 
\label{th-el}
\end{theorem}
We are interested in understanding the relationship between $R(L)$ and the bimonster.
Daniel Allcock has made a conjecture to explain this relationship using complex hyperbolic
geometry. (In-fact the conjecture predates \cite{TB:EL}).
\begin{conjecture}[Allcock, see \cite{DA:Monstrous}, \cite{TB:EL}]
\label{c-MP}
{\it Consider the action of $\Gamma = \PP \Aut(L)$ on the complex hyperbolic space $Y$.
Let $\cM$ be the union of the fixed points of the reflections in $\Gamma$. Let
$Y^{\circ} = Y \setminus \cM$. Then the orbifold fundamental group of
$Y^{\circ}/\Gamma$ maps onto bimonster.}
\end{conjecture}
See \cite{DA:Monstrous} for more on this conjecture and its possible ramifications. There are
speculative ideas explored in \cite{DA:Monstrous} regarding a possible candidate for the conjectured
monster manifold and an interpretation of $Y/\Gamma$ as a nice moduli space.
\end{topic}
%
%*******************************************************************************************************
%
\begin{topic}{\bf Summary of results: }  
\label{t-summary}
In section \ref{sec-general}, we recall some basic definitions and set up our notations.
We explain how to construct our lattice $L$ from incidence geometry of $\PP^2(\FF_3)$.
\par
Section \ref{sec-L} recalls the results from \cite{TB:EL} that we need, in some detail.
Lemma 3.2 and theorem 5.8 of \cite{TB:EL} show that $\Aut(L)$ can be generated
by sixteen  complex reflections of order $3$, making an $M_{6 6 6}$ diagram. 
In section \ref{sec-L}, we improve this by  showing that fourteen complex reflections of order $3$
suffice to generate the reflection group of $L$. Since $L$ is $14$ dimensional, these fourteen
reflections, making a $M_{6 5 5}$ diagram, form a minimal set of generators for
$\Aut(L)$.
\par
Section \ref{sec-pi1} contains the bulk of our work. Here we study the fundamental group of the orbifold
$Y^{\circ}/\Gamma$, that we denote by $G$. By definition of $G$ (see \ref{def-G}), there is
an exact sequence,
\begin{equation}
 1 \to \pi_1(Y^{\circ}) \to G \xrightarrow{\pi^G_{\Gamma}} \Gamma \to 1.
\label{exact}
\end{equation}
By the results from \cite{TB:EL} quoted above, we have a surjection $\phi$ from $\cA(D)$ to $\Gamma$.
The main result of section \ref{sec-pi1}  is theorem \ref{th-lift}, which says that the map $\phi$ lifts to a
homomorphism $\psi: \cA(D) \to G$. This answers a question asked by Allcock in \cite{DA:Monstrous}.
\par
To put this theorem in context, we should remark that we expect the fundamental group $G$ to be
something like the Artin group $\cA(D)$, maybe with some extra relations.
This expectation is based on the the analogy of our reflection group $R(L)$
with the Weyl groups and the theorem of Brieskorn, Saito \cite{BK:AG} and Deligne \cite{PD:LI}
on fundamental groups of complements of complex hyperplane arrangements
associated to Weyl groups.
The analogy between the reflection group $R(L)$ and Weyl groups has been a
useful guiding principle in this project.
\par
Theorem \ref{th-lift} is a step towards proving conjecture \ref{c-MP}. To explain this,
let $\tilde{Y}^{\circ}$ be the universal cover of $Y$. The group $G$ acts as Deck transformations
on the ramified covering $(\tilde{Y}^{\circ} \to Y^{\circ}/\Gamma)$. 
Let $G_1 \subseteq G$ be the image of the map $\psi$
\footnote{We expect that $G_1 = G$. 
It seems that the method of proving of theorem \ref{th-lift} might also be useful for proving
$G_1 = G$.}.
Let $N$ be the normal subgroup of $G_1$ generated by $\lbrace \psi(r)^2 \colon r \in D \rbrace$.
Then we have the following tower of ramified covering:
\newline
{\small
\centerline{
 \xymatrix @!=1pc { 
& \tilde{Y}^{\circ}  \ar@{->}[dl] \ar@{->}[dr] & \\
Y^{\circ} \ar@{->}[dr] & & \tilde{Y}^{\circ}/N \ar@{->}[dl] \\ 
& Y^{\circ}/\Gamma &
}}}
The following observations should justify the effort needed to prove the theorem above.
\begin{enumerate}
 \item By theorem \ref{th-lift}, there is a surjection from $\Cox(D,2)$ to $G_1/N$.
If the ``deflation relations'' hold in $G_1/N$ and $G_1/N$ has more
than two elements then $G_1/N$ is the bimonster. On the other hand, $G_1/N$ acts
naturally as deck transformations on the (possibly ramified) covering 
$(\tilde{Y}^{\circ}/N \to Y^{\circ}/\Gamma)$. 
\item In addition to the above, if the map $\psi$ is onto, that is, if $G_1 = G$, then that
proves Allcock's conjecture. If $G_1 = G$, then the deflation relation  probably holds in
$G/N$. Sketch of a geometric argument to prove this was explained to me by Daniel Allcock.
\end{enumerate}
We prove theorem \ref{th-lift} by constructing explicit homotopies
between paths in $Y^{\circ}$ to obtain relations in the group $G$.
Let $\gamma_1$ and $\gamma_2$ be two paths in $Y^{\circ}$ with the same
beginning and endpoints.
To construct a homotopy from $\gamma_1$ to $\gamma_2$ in $Y^{\circ}$, we construct
a $2$--cell $C' \subseteq Y$ whose boundary is $\gamma_1 \cup \gamma_2$
and then show that $C'$ does not intersect the mirrors of reflection. 
The main tool to show that $C'$ avoids the mirrors, is to use theorem \ref{th-el}(c),
which provides some information about the configuration of mirrors near the point
$\br \in Y^{\circ}$ fixed by the group $2. \op{PGL}(3, \FF_3)$ of diagram automorphisms.
This information about the mirror arrangement, together with some complex
hyperbolic geometry, lets us restrict the possible set of mirrors intersecting $C'$ to a finite set.
The proof is completed by directly checking that the remaining finite set of mirrors do not intersect
$C'$.
\par
The unique point $\br$ fixed by the group $Q \simeq 2.\op{PGL}(3, \FF_3)$ of diagram automorphisms, plays an
important role in all the major arguments in \cite{TB:EL} as well as the present article. In view
of this, in section \ref{sec-fixed}, we make a detailed study of the two dimensional lattice $F$
 fixed by the group $\op{PGL}(3, \FF_3) \subseteq Q$.
The complex hyperbolic line $\PP_+(F^{\CC})$ is isometric to the upper half plane $\cH^2$.
We construct an explicit isometry $\beta: \PP_+(F^{\CC}) \to \cH^2$ such that
``cusps of Leech type'' map to $0$ and $\infty$ and
$\br$ maps to $\sqrt{-1}$. Let $\Gamma_F$ be the set of elements of $\Gamma$ that fix $F$ as a set. 
The isometry $\beta$ induces a group homomorphism $c_{\beta}: \Gamma_F \to \op{PSL}_2(\RR)$.
We show that the image of $c_{\beta}$ contains the principal congruence subgroup $\Gamma(13)$.
The group $c_{\beta}(\Gamma_F) \cap \op{PSL}_2(\ZZ)$ has genus zero. In-fact it is a conjugate of
$\Gamma_0(13)$. The diagram automorphism $\sigma$, that correspond to interchanging points and lines
of $\PP^2(\FF_3)$, acts as the involution $(\tau \mapsto -\tau^{-1})$.
\par
This result provides an explicit way to obtain ordinary modular forms from automorphic
forms of type $U(1,13)$ defined on the hermitian symmetric space $Y$. Let $f$ be a meromorphic
automorphic form on $\CC H^{13}$ automorphic with respect to the group $\Gamma$, with zeroes
and poles along the mirrors of reflection. 
From \ref{th-el}(c), we know that the complex hyperbolic line $\PP_+(F^{\CC})$
 is not contained in any mirror, so the restriction of $f$ to $\PP_+(F^{\CC})$ is a well-defined and non-zero,
thus it is a meromorphic modular form of level thirteen. Meromorphic automorphic forms on $Y$ with zeros and poles along the mirrors of reflection can be
constructed by the Borcherds method of singular theta lift (see \cite{REB:Aut} for the general
method and \cite{DA:Leech} for our example). Alternatively, as we shall in \ref{remark-aut},
one can define such automorphic forms as explicit infinite series similar to the Eisenstein series
(also see \cite{EL:CA}, where these are called the Poincare-Weierstrass Series).
These automorphic forms may be useful in obtaining an explicit projective uniformization of
$Y^{\circ}/\Gamma$.
\end{topic}
{\bf Acknowledgments: } I want to thank prof. Jon Alperin, Prof. Richard Borcherds, Prof. John Conway,
Prof. Benson Farb, Prof. George Glauberman, Prof. Stephen Kudla and Prof. John Mckay for encouragement
and useful conversations. Most of all, I want to thank Prof. Daniel Allcock for generously sharing his
insights and ideas on this project.
\begin{topic}{\bf Index of some commonly used notations: } We use the Atlas notation for groups.
\begin{tabbing}
$\cA(D)$ X\= the Artin group of the diagram $D$.X\=\kill
$\cA(D)$ \> the Artin group of the diagram $D$.\\
$\Aut(L)$\> automorphism group of the lattice $L$. \\
$\beta$ \> an isometry from $\PP_+(F^{\CC})$ to $\cH^2$ or the matrix
$\bigl( \begin{smallmatrix} p & \omega \\ 1  & \bar{p} \end{smallmatrix} \bigr)$ that represent it.\\
$\CC H^n$\> the complex hyperbolic space of dimension $n$. \\  
$D$      \>  the incidence graph of $P^2(\mathbb{F}_3)$ or the set of $26$ simple roots of $L$ labeled by\\ 
         \> vertices of this graph. \\ 
$D(K)$      \> the discriminant group of a lattice $K$; i.e., $D(K) = K'/K$. \\          
$\cE$    \>  $=\ZZ[e^{2 \pi i/3}] $. \\
$\epsilon$\> a small positive real number. \\
$F$ \>   the sub-lattice of $L$ fixed by the group $\op{PGL}(3, \FF_3)$ of diagram automorphisms. \\
$G$      \> the orbifold fundamental group of $X^{\circ}$. \\
$\Gamma$ \> the automorphism group of $L$ modulo scalars: $\Gamma = \PP \Aut(L)$. \\
$\op{ht}(r)$ \>  height of a vector $r$, given by $\op{ht}(r) = \abs{\ip{ \br}{ r }}/\abs{\br}^2$.\\
$l$      \>  an element of $\cL$. \\
$L$      \>  the complex Leech lattice plus a hyperbolic cell, defined over $\cE$. \\
$L'$     \>  the dual lattice of $L$.\\
$\cL$    \>  the set of lines of $\PP^2(\FF_3)$ or the simple roots of $L$ that correspond to them.\\
%$\op{L}_3(3)$\> $= \op{PSL}_3(\FF_3) = \op{PGL}_3(\FF_3)$.\\
$\cM$    \>  the union of the mirrors of the reflection group of $L$.\\
$p$      \> $=2 + \omega$.\\
$p_1$      \> $= 3 - \omega$.\\
$\cP$    \>  the set of points of $\PP^2(\FF_3)$ or the simple roots of $L$ that correspond to them.\\
$\phi_r$ \>  $\omega$--reflection in the vector $r$.\\ 
$Q$      \> the group of diagram automorphisms acting on $Y$; one has $Q \simeq 2.\op{PGL}(3, \FF_3)$. \\
$r_i$    \> $r_1, \dotsb, r_{26}$ are the simple roots. \\
$\rho_i$ \> $\rho_i = r_i$ if $r_i \in \cP$, and $\rho_i = \xi r_i$ if $r_i \in \cL$.  \\
$R(L)$   \>  (complex) reflection group of the lattice $L$.\\
$[\br]$  \> the unique point in the complex hyperbolic space $Y$, fixed by $Q$. (see equation \eqref{eq-defrho}). \\
$\sigma$ \> A diagram automorphism that corresponds to interchanging the points and lines \\
         \> of $\PP^2(\FF_3)$. \\
$\theta$ \> $= \omega - \bar{\omega}$.\\
$\omega$ \>  $e^{2\pi i/3}$. \\
$[w_{\cP}]$\> The point of $Y$ where the $13$ mirrors $\lbrace x^{\bot} : x \in \cP\rbrace$ meet. (see equation \eqref{defwpwl})\\ 
$[w_{\cL}]$\> The point of $Y$ where the $13$ mirrors $\lbrace l^{\bot} : l \in \cL\rbrace$ meet. (see equation \eqref{defwpwl})\\
$x$      \>  an element of $\cP$.\\
$\xi$    \> $= e^{-\pi i/6}$.\\
$X^{\circ}$ \> the orbifold  $Y^{\circ}/\Gamma$.\\
$Y$      \> the set of complex lines of positive norm in $ L \otimes_{\cE} \CC$. (Note: $Y \simeq \CC H^{13}$.)\\
$Y^{\circ}$ \> $= Y \setminus \cM$.\\
\end{tabbing}
\end{topic}

%
%*******************************************************************************************************
%
%
%*******************************************************************************************************
%
\section{Some generalities on complex hyperbolic reflection groups}
\label{sec-general}
%
%*******************************************************************************************************
%
\begin{topic}{\bf The complex hyperbolic space: }
A general reference for complex hyperbolic geometry is \cite{WG:CHG}.
Let $V$ be a complex vector space, with an Hermitian form $\ip{\;}{\;}$.
Assume that $V$ is Lorentzian, that is, the Hermitian form on $V$ has signature $(1,m-1)$.
Then the open subset $\PP_+(V)$ of the projective space $\PP(V)$, consisting
of the complex lines of positive norm, is called the complex hyperbolic space
of $V$. If $H \subseteq V$, then let $[H]$ be the subset of $\PP_+(V)$ or $\PP(V)$ 
determined by $H$. We shall often abuse notation and write $H$ for $[H]$, if there
is no possibility of confusion.
\par
Let $\CC^{m,n}$ denote the complex vector space $\CC^{m+n}$ with the Hermitian form
\begin{equation*}
\ip{z}{w} = \sum_{i = 1}^m \bar{z}_i w_i - \sum_{i = m+1}^{m+n} \bar{z}_i w_i.
\end{equation*}
The $n$ dimensional complex hyperbolic space $\CC H^n = \PP_+(\CC^{1,n})$ is
homeomorphic to the unit ball $B^n(\CC) \subseteq \CC^n$ via the isomorphism
$b': [z_0, \dotsb, z_n]' \mapsto (z_1/z_0, \dotsb, z_n/z_0)$. Let $x$ and $y$
be two vectors in $V$ having positive norm. The metric on $\CC H^n$ is given by 
\begin{equation*}
d([x],[y]) = \cosh^{-1} \Bigl(  \frac{\abs{ \ip{x}{y}} }{\abs{x}.\abs{y}} \Bigr).
\end{equation*}
This differs from the metric used in \cite{WG:CHG} by a factor of 2, but the scaling of the metric
is not important for us.
\end{topic}
\begin{topic}{\bf The complex hyperbolic line and the real hyperbolic plane:}
We need to consider two models of the real hyperbolic plane, namely the
unit ball $B^1(\CC)$ and the upper half plane $\cH^2$. It will be convenient to
identify both as subsets of $\PP^1(\CC)$ by $\tau \mapsto 
\bigl( \begin{smallmatrix} \tau \\ 1 \end{smallmatrix} \bigr)$.
\par
The map $b': [z_0, z_1] \to z_1/z_0$  defines an isometry between
the complex hyperbolic line $\CC H^1$ and the (ball model of) real hyperbolic
plane.
Let $\cH^2$ be the upper half plane with the Poincare metric:
$d(\tau, \tau') 
= \cosh^{-1}\bigl(2^{-1} \op{Im}(\tau)^{-\tfrac{1}{2}} \op{Im}(\tau')^{-\tfrac{1}{2}}\abs{\tau' - \bar{\tau}} \bigr)$.
Let $C: B^1(\CC) \to \mathcal{H}^2$ denote the Cayley isomorphism: 
$C(w) =i\bigl( \tfrac{ 1 + w}{1  - w} \bigr) $.
The composition 
\begin{equation*}
b = C \circ b' : \PP_+(\CC^{1,1} ) \to \cH^2
\end{equation*}
is an isometry.
Let $S = \bigl( \begin{smallmatrix} 0 & -1 \\  1 & 0 \end{smallmatrix} \bigr)$
and $T = \bigl( \begin{smallmatrix} 1 & 1 \\  0 & 1 \end{smallmatrix} \bigr)$
be the standard generators of $\op{SL}_2(\ZZ)$. 
\end{topic}
\begin{topic}{\bf Complex lattices and their reflection groups: }
We recall our notations regarding complex reflection groups. In most places,
we maintain the notations of \cite{TB:EL} and refer the reader to section 2.2 of \cite{TB:EL}
for more details. Let $\xi = e^{-\pi i/6}$, $\omega = -\xi^2$ and $\cE = \ZZ[\omega]$. An
{\it $\cE$--lattice $K$} is a free $\cE$--module of finite rank, together with an
$\cE$--valued Hermitian form $\ip{\;}{\;} : K \times K \to \cE$ (always linear
in the second variable). Let $K^{\CC} = K \otimes_{\cE} \CC$ be the underlying Hermitian
vector space of $K$. We shall mainly be concerned with a lattice $L$, for which
$ L^{\CC} \simeq \CC^{1, 13}$. So the definite lattices that we consider will always
be negative definite. Let $K'$ be the {\it dual lattice} of $K$, defined by
$K' = \lbrace v \in K^{\CC} \colon \ip{v}{x} \in \cE \text{\; for all \;} x \in K \rbrace$. 
Let $\op{D}(K) = K'/K$ be the {\it discriminant group} of $K$.
\par 
Let $r$ be a primitive vector of $K$ having negative norm, that is,
$\abs{r}^2 := \ip{r}{r} < 0$. Let $\alpha$ be a root of unity in $\cE$, $\alpha \neq 1$.
A {\it complex reflection} $\phi_r^{\alpha} \in \Aut(K)$ is an automorphism of $K$ that
fixes the hyperplane $r^{\bot}$ orthogonal to $r$ and multiplies $r$ by $\alpha$.
The vector $r$ is called the {\it root} of the reflection and the hyperplane $r^{\bot}$
(or its image in the projective space $\PP(K^{\CC})$) is called the {\it mirror} of
the reflection. The {\it reflection group} of $K$, denoted by $R(K)$, is the subgroup
of $\Aut(K)$ generated by reflections in the roots of $K$.
We write $\phi_r = \phi_r^{\omega}$ and call it the $\omega$--reflection in $r$.
\end{topic}
%
%
%*******************************************************************************************************
%
\begin{topic}{\bf The incidence graph of $\PP^2(\FF_3)$: }
\label{DtoL}
The projective plane over $\FF_3$ has $13$ points and $13$ lines.
Let $\cP$ be the set of points and $\cL$ be the set of lines.
If a point $x \in \cP$ is incident on a line $l \in \cL$, then we write $x \in l$.
Let $D$ be the (directed) incidence graph of $\PP^2(\FF_3)$.
The vertex set of $D$ is $\cP \cup \cL$. There is a directed edge in $D$ from 
a vertex $l$ to a vertex $x$ if $x \in \cP$, $l \in \cL$ and $x \in l$.
\end{topic}
\begin{topic}{\bf Definition of the Lorentzian lattice from geometry of $\PP^2(\FF_q)$: }
\label{defDtoL}
Let $p = 2 + \omega$.
Let $L^{\circ}$ be the free $\cE$--module of rank $26$ with basis vectors
indexed by $D =\cP \cup \cL$. Let $x,x' \in \cP$ and $l,l' \in \cL$. Define a Hermitian
for on $L^{\circ}$ by 
\begin{align}
\ip{x}{x'} = \begin{cases} 
-3 & \text{\;if \;\;} x = x', \\
0 & \text{otherwise.} 
\end{cases}
& &
\ip{l}{l'} = \begin{cases} 
-3 & \text{\;if \;\;} l = l', \\
0 & \text{otherwise.} 
\end{cases}
& &
\ip{x}{l} = \begin{cases} 
p & \text{\;if \;\;} x \in l, \\
0 & \text{otherwise.} 
\end{cases}
\label{ipd}
\end{align}
The lattice $L^{\circ}$ has a $12$ dimensional radical. To see this, let 
$w_l  = \bar{p} l + \sum_{ x \in l} x$ for any  $l \in \cL$.
Then using \eqref{ipd} and the geometry of $\PP^2(\FF_3)$ one easily checks that
\begin{equation} 
\ip{x'}{w_l} = 0 \text{\;\; and \;\;}  \ip{w_l}{l'} = p \text{\;\; or all \;\;} x' \in \cP \text{\; and \;} l' \in \cL.
\label{eq-wl}
 \end{equation}
So if $l_1, l_2 \in \cL$, then 
$(w_{l_1} - w_{l_2})$ is in the radical $\op{Rad}(L^{\circ})$. So $\op{rank}(\op{Rad}(L^{\circ})) \geq 12$.
Define
\begin{equation*} 
L_3 = L^{\circ}/\op{Rad}(L^{\circ}).
 \end{equation*}
The Hermitian form on $L^{\circ}$
descends to a Hermitian form on $L_3$, which is again denoted by $\ip{\;}{\;}$.
Let $w_{\cP}$ be the image of $w_l$ in $L_3$ (for any $l \in \cL$).
Equation \eqref{eq-wl} implies that the vector $w_{\cP}$ is orthogonal (in $L_3$)
to each $x \in \cP$ and $\ip{w_{\cP}}{ l } = p$ for all $l \in \cL$.
It follows that the matrix of dot products of the vectors
$\cP \cup \lbrace w_{\cP} \rbrace$ is a diagonal matrix with diagonal entries 
$( -3, -3, \dotsb, -3, 3)$. So $\op{rank}(L) \geq 14$.
It follows that $\op{rank}(\op{Rad}(L^{\circ})) = 12$, $\op{rank}(L_3) = 14$ and 
$L_3$ is a non-degenerate $ \cE$--module of signature
$(1,13)$.
\end{topic}
\begin{remark}
From the properties of the incidence matrix of $\PP^2(\FF_3)$ one can prove that
$L_3$ is $p$-modular, that is $p L_3' = L_3$. The definition of $L_3$ given in \ref{defDtoL}
can be mimicked for finite projective planes $\PP^2(\FF_q)$
for other values of $q$ and number rings other than $\cE$.
This way one obtains infinitely many $\sqrt{-q}$-modular Lorentzian Hermitian lattices, see
\cite{TB:ML}. If $q \equiv 3 \bmod 4$ is a rational prime and
$(q^2 + q + 1)$ is a norm of some element in $\QQ[\sqrt{-q}]$, then
we get a even unimodular $2q(q+1)$ dimensional $\ZZ$-lattice
whose symmetry group contains $\op{PGL}(3,\FF_q)$.
The first one among these lattices is the Leech lattice.
\end{remark}
%
%*******************************************************************************************************
%
%
%*******************************************************************************************************
%

\section{The reflection group of the Lorentzian Leech lattice}
%
%*******************************************************************************************************
%
%
%*******************************************************************************************************
%
\label{sec-L}
Chiefly to set up notation, we need to recall a few results from \cite{TB:EL}
in some detail. Then we prove a couple of small lemmas. As before,
let $L$ be the direct sum of the complex Leech lattice and a hyperbolic cell.
Let $L_3$ be the lattice defined in \ref{defDtoL}.
\begin{topic}{\bf Observation:} {\it Since both $L$ and $L_3$ 
have signature $(1,13)$ and satisfy $pL' = L$ and $p L_3' = L_3$, it follows that $L \simeq L_3$.
This can also be directly proved by showing that the lattice $L$ contains $26$ vectors of norm $-3$
(called the {\it simple roots}) that have inner products prescribed by \eqref{ipd} (see 3.1 of \cite{TB:EL}). 
 }
\end{topic}
We want to study the action of $\Gamma = \PP \Aut(L)$ on $ Y = \PP_+(L^{\CC}) \simeq \CC H^{13}$.
Note that the simple roots of $L$ correspond to the vertices of the graph $D$.
A vertex of $D$ and the corresponding simple root of $L$ are denoted by
the same symbol. If $\lbrace r,s \rbrace$ is an edge of $D$, then the $\omega$-reflections
in the simple roots $r$ and $s$ braid. If $\lbrace r,s \rbrace$ is not an edge, then the reflections
commute. As mentioned in \ref{th-el}(a), the order three complex reflections in the $26$
simple roots generate $\Aut(L)$.
\par
The construction of $L$ given in \ref{defDtoL} implies that the group of {\it diagram
automorphisms} $Q \simeq 2. \op{PGL}(3, \FF_3)$ acts on the complex hyperbolic space $Y = \PP_+(L^{\CC})$.
The action of $\op{PGL}(3, \FF_3)$ on $L$ point-wise fixes a two dimensional primitive sub-lattice
$F$ of $L$ having a basis $w = ( w_{\cP}, w_{\cL} )$ where
\begin{equation}
w_{\cP}  = \bar{p} l + \sum_{ x' \in l} x'  \;\;\text{and} \;\;\;\;
w_{\cL}  = p x + \sum_{ x \in l'} l', 
\label{defwpwl}
\end{equation}
for any $x \in \cP$ and $l \in \cL$.
We shall make a detailed study of the lattice $F$ in section \ref{sec-fixed}.
From the inner products given in \eqref{ipd}, it is easy to check that for all $x \in \cP$ and $l \in \cL$,
we have,  
\begin{equation}
 \ip{w_{\cP}}{x} = \ip{w_{\cL}}{l} = 0, \text{\;} \ip{w_{\cP}}{l} = \ip{x}{w_{\cL}} = p, \text{\; and \;}
\bigl( \begin{smallmatrix} \ip{w_{\cP}}{w_{\cP}} & \ip{w_{\cP}}{w_{\cL}} \\ \ip{w_{\cL}}{w_{\cP}} & \ip{w_{\cL}}{w_{\cL}} \end{smallmatrix} \bigr)
 = \bigl( \begin{smallmatrix} 3 & 4p \\ 4 \bar{p} & 3 \end{smallmatrix} \bigr).
\label{eq-ipwpwl}
\end{equation}
Let $D^{\bot}$ be the set of mirrors perpendicular to the simple roots. These are called the
{\it simple mirrors}. Equation \eqref{eq-ipwpwl} implies that the thirteen simple mirrors in
$\CC H^{13}$ corresponding to the elements of $\cP$ (resp. $\cL$) meet at $[w_{\cP}]$ (resp. $[w_{\cL}]$).
There exists $\sigma \in \Aut(L)$ (cf. \cite{TB:EL}, Section 5.3) that corresponds to
interchanging points and lines of $\PP^2(\FF_3)$. 
The group $Q \subseteq \Gamma$ of diagram automorphisms is generated by $\sigma$ and
$\op{PGL}(3, \FF_3)$. Let 
\begin{equation}
\br = \tfrac{1}{26}\bigl(\sum_{x \in \cP} x + \xi \sum_{l \in \cL} l\bigr) 
= \frac{w_{\cP} + \xi w_{\cL}}{ 2(4 + \sqrt{3})}.
\label{eq-defrho}
\end{equation}
Then $[\br]$ is the midpoint of the geodesic joining $[w_{\cP}]$ and $[w_{\cL}]$.
The diagram automorphism $\sigma$ interchanges $[w_{\cP}]$ and $[w_{\cL}]$, so $[\br]$ is the only
point in $\CC H^{13}$ fixed by $Q$. Since $Q$ fixes $[\br]$ and acts transitively on the simple
mirrors, it follows that $[\br]$ is equidistant from the simple mirrors. Let
$d_0 = d(r^{\bot}, [\br])$ for $r \in D$.
As mentioned in theorem \ref{th-el}(c), we know that $d(r^{\bot}, [\br]) \geq d_0$
for all root $r$ of $L$ and $d(r^{\bot}, [\br]) = d_0$ if and only if $r^{\bot}$ is a simple
mirror. In other words, the simple mirrors are closest to $[\br]$ and all other mirrors are
further away.
\begin{lemma}
Consider $ \Gamma = \PP \Aut(L)$ acting on $Y \simeq \CC H^{13}$. 
Then the stabilizer (in $\Gamma$) of the set $D^{\bot}$ and the stabilizer
of $[\br]$ are both equal to $Q \simeq 2.\op{PGL}(3, \FF_3)$.
\end{lemma}
\begin{proof}
If $g \in \PP \Aut(L)$ fixes $D^{\bot}$ as a set, then it must fix $[\br]$,
it being the only point, equidistant from each simple mirror.
Conversely, if $g \in G$ fixes $[\br]$ then it must permute the mirrors in
$D^{\perp}$ in some way and thus determine an element of $2.\op{PGL}(3, \FF_3)$.
So the stabilizer of the set $D^{\perp}$ and the point $[\br]$
are equal.
\par
Suppose $g$ acts trivially on $D^{\bot}$.
Let $\tilde{g}$ be a lift of $g$ in $\Aut(L)$. 
For each $r \in D$, it follows that $\tilde{g} r = \mu_r r$ for some root of unity
$\mu_r$. If $r$ and $s$ are two simple roots with $\ip{r}{s}  = p$,
then $p = \ip{\tilde{g} r}{ \tilde{g} s} = \bar{\mu}_r \mu_s \ip{r}{s} = \bar{\mu}_r \mu_s p$.
So $\mu_r = \mu_s$ whenever there is an edge between $r$ and $s$ in the graph
$D$, which implies $\mu_r = \mu$ is a constant. So $g$ is equal to the identity.  
\end{proof}
In \cite{TB:EL}, we showed that $\Aut(L)$ is generated by the $16$ simple reflections
in the roots $\lbrace a, b_j, c_j, d_j, e_j, f_j \colon j = 1, 2, 3 \rbrace$ that form
the $M_{666}$ diagram. We shall end this section by noting a small improvement of this
fact.
%
%*******************************************************************************************************
%
\begin{lemma}
The automorphism group of $L$ is generated by the fourteen $\omega$-reflections in the simple roots
$a, f_1$ and $b_i,c_i,d_i, e_i$ for $i = 1, 2, 3$. 
Since $L$ is $14$ dimensional, the $14$ simple reflections of a $M_{6 5 5}$ diagram form a minimal set of
generators for $\Aut(L)$.  
\label{lemma-14}
\end{lemma}
The proof of the lemma uses the deflation relations, which we now recall. By a sub-graph of $D$
we mean the graph formed by taking a subset of vertices of $D$ and all the edges between these
vertices in $D$. A $12$-gon in $D$ is a sub-graph of $D$ that has the shape of a circuit of length
$12$ (in other words, an affine $A_{11}$ diagram),
Let $y$ be a $12$-gon in $D$ and  let $\lbrace y_1, \dotsb, y_{12} \rbrace$
be the successive vertices of $y$. Consider the relation
\begin{equation}
({y_1} y_2 \dotsb y_{10}) y_{11} ({y_1} y_2 \dotsb y_{10})^{-1} = y_{12} 
\label{eq-deflate}
\end{equation}
Following \cite{CCS:26}, we call this relation ``$\op{deflate}(y)$''.
The affine Coxeter group of type $A_{11}$ generated by $y_1,\dotsb, y_{12}$ reduces to the symmetric
group $S_{12}$ in presence of the relation $\op{deflate}(y)$.
The bimonster is the quotient of $\Cox(D,2)$ obtained by adding the relations
$\op{deflate}(y)$ for all free $12$-gon $y$ in $D$ (see \cite{CCS:26}).
%
%*******************************************************************************************************
%
\begin{proof}[Proof of lemma \ref{lemma-14}]
Take the $12$-gon $(y_1, \dotsb, y_{12}) = (f_2,e_2,d_2,c_2,b_2,a,b_1,c_1,d_1,e_1,f_1, a_3)$ in $D$.
(Here, only for this proof, we are using the names for the simple roots given in \cite{CCS:26}.) 
One can check that
\begin{equation*}
\phi_{f_2} \phi_{e_2} \phi_{ d_2} \phi_{ c_2} \phi_{ b_2} \phi_{ a} \phi_{ b_1} \phi_{ c_1} \phi_{ d_1} \phi_{ e_1}( f_1)
=\omega^2 a_3.
\end{equation*}
So equation \eqref{eq-deflate} holds for this $12$-gon. It is an amusing exercise to show that
the group $2.\op{PGL}(3, \FF_3)$ acts transitively on the set of marked $12$-gons
in $D$. So the relation \eqref{eq-deflate} holds for any $12$-gon in $D$. 
The deflation relation implies that
\begin{equation*}
y_{12} = (y_1 \dotsb y_{10}) y_{11} (y_1 \dotsb y_{10})^{-1} = (y_{11} \dotsb y_2) y_1 (y_{11} \dotsb y_2)^{-1}.
\end{equation*}
Moving the $y_1$'s to the other side of the equation, and using the commuting relations, we get
\begin{align*}
(y_2 \dotsb y_{10}) y_{11} (y_2 \dotsb y_{10})^{-1}
&= y_1 ^{-1} (y_{11} \dotsb y_2) y_1 (y_{11} \dotsb y_2)^{-1} y_1 \\
&= (y_{11} \dotsb y_3) y_1^{-1} y_2 y_1 y_2^{-1} y_1 (y_{11} \dotsb y_3)^{-1}.
\end{align*}
Using the braiding relations $y_1 y_2 y_1 = y_2  y_1 y_2 $ and $y_i^{-1} = y_i^2$ we get
\begin{equation*}
y_1^{-1} y_2 y_1 y_2^{-1} y_1 = y_1^{-1} y_2 y_1 y_2 y_2 y_1 = y_1^{-1} y_1 y_2 y_1 y_2 y_1 = y_2 y_1 y_2 y_1 =y_2^{-1} y_1 y_2.
\end{equation*}
It follows that
\begin{equation*}
(y_{11} \dotsb y_3) y_2^{-1} y_1 y_2 (y_{11} \dotsb y_3)^{-1}
=(y_2 \dotsb y_{10}) y_{11} (y_2 \dotsb y_{10})^{-1}.
\end{equation*}
So $y_1$ can be expressed in terms of $y_2, \dotsb, y_{11}$.
So the reflections $\phi_{f_2}$ and $\phi_{f_3}$ can be expressed in terms of
of the fourteen reflections stated in the lemma.
\end{proof}
%
%*******************************************************************************************************
%
%
%*******************************************************************************************************
%
\section{The fundamental group of the mirror complement-quotient}
%
%*******************************************************************************************************
%
%
%*******************************************************************************************************
%
\label{sec-pi1}
In this section we shall show that the Artin group $\cA(D)$ maps to the orbifold
fundamental group of $Y^{\circ}/\Gamma$. So there is an action of $\cA(D)$ on the
universal cover of $Y^{\circ}$ and an action of $\Cox(D,2)$ on a (possibly ramified)
cover of $Y^{\circ}/\Gamma$.
\begin{topic} {\bf Some basics on complex hyperbolic space for estimating distances: } 
\label{distbasic}
For this sub-section, we let $V$ be a general $(n+1)$ dimensional complex vector space with
a Hermitian form of signature $(1,n)$. At the end of the sub-section we shall go back to our example.
Let $V^+$ be the set of vectors of strictly positive norm in $V$ and $\PP_+(V) \simeq \CC H^n$
be the complex hyperbolic space of $V$.
Whenever we talk of distance in this section, it is with reference to the metric on $\CC H^n$.
Given $x, y$ in $\CC H^n$, let $\Geod(x,y)$ be the real geodesic segment joining $x$ and $y$. 
\par
In the following, let $x_1,  \dotsb, x_n, x, y, z, w$ be elements of $V^+$. Assume further that
$x_1, \dotsb, x_n$ lie in a totally real subspace, that is, $\langle x_i, x_j \rangle \in \RR_+$
for $1 \leq i, j \leq n$. Then, for all $i \neq j$ the Euclidean straight line segment in $V^+$
joining $x_i$ and $x_j$ (denoted by $\Conv(x_i, x_j)$) determines the real geodesic segment in
$\CC H^n$ joining $x_i$ and $x_j$, that is, $[\Conv(x_i, x_j)] = \Geod([x_i], [x_j])$.
Let $\Conv(x_1, \dotsb, x_n)$ be the convex hull of $x_1, \dotsb, x_n$ in $V^+$:
\begin{equation*}
\Conv(x_1, \dotsb, x_n) 
= \lbrace t_1 x_1 + \dotsb + t_n x_n \colon 0 \leq t_j \leq 1, \sum t_j = 1 \rbrace.
\end{equation*}
The set $\Conv(x_1, \dotsb, x_n)$ determines a totally real, totally geodesic subset of $\CC H^n$.
Given two non-empty subsets $Z$ and $W$ of $\CC H^n$ with $W$ compact, let 
\begin{equation*}
\md_Z(W) = \max \lbrace d(Z, w) \colon w \in W \rbrace.
\end{equation*}
The following observations will be useful for our computation. (Recall our assumption: $\ip{x_i}{x_j} \in \RR_+$).
\par
\begin{enumerate}
\item Let $\lbrace W_j: j \in J \rbrace$ be a (possibly infinite) collection of compact sets
such that $\cup_{j \in J} W_j$ is compact. Then
\begin{equation}
\md_H( \cup_{j \in J} W_j  ) = \sup \lbrace \md_H(W_j) \colon j \in J \rbrace
\label{mdunion} 
\end{equation}
\item
By general properties of negatively curved spaces, one knows that
$\max \lbrace d(z, w) \colon w \in \Conv(x_1, x_2) \rbrace$
 is attained when $w$ is either $x_1$ or $x_2$, in other words 
\begin{equation*}
\md_z( \Conv(x_1, x_2) ) = \max \lbrace d(z, x_1), d(z, x_2) \rbrace.
\end{equation*}
\item Let $ W = \Conv(x_2, \dotsb, x_n) $.
Using $\Conv(x_1, x_2, \dotsb, x_n) = \cup_{w \in W} \Conv(x_1,w)$, it follows that
\begin{equation*}
\md_z(\Conv(x_1, \dotsb, x_n)) 
= \sup \lbrace \md_z(\Conv(x_1, w)) \colon w \in W \rbrace 
= \max \lbrace d(z, x_1) , \md_z(W) \rbrace.
\end{equation*}
By induction on $n$, it follows that 
$\max \lbrace d(z, w) \colon w \in  \Conv(x_1, x_2, \dotsb, x_n) \rbrace$ is
attained when $w = x_i$ for some $i$. 
\item 
Let $H$ be a subset of $\CC H^n$ and $\Delta = \Conv(x_1, x_2, x_3)$. Let 
\begin{equation*}
\delta_1 = d(x_1, H) + \max \lbrace d(x_1, x_2), d(x_1, x_3) \rbrace.
\end{equation*}
Define $\delta_2, \delta_3$ similarly by cyclic permutation of $x_1, x_2, x_3$.
Choose $w_0 \in \Delta$ be such that $\md_H(\Delta) = d(w_0, H)$. Then, from the previous
remarks, we have,
\begin{equation*}
\md_H(\Delta) = d(w_0, H) \leq d(x_1, H) + d(x_1, w_0) \leq d(x_1, H) + \md_{x_1}(\Delta) \leq \delta_1.
\end{equation*}
It follows that 
\begin{equation}
\md_H(\Delta) \leq \min \lbrace \delta_1, \delta_2, \delta_3 \rbrace.
\label{mdHD}
\end{equation}
\item Let $x \in V^+$ and $H$ be a complex linear subspace in $V$ which meets
$V^+$. Then $H$ determines a totally geodesic subspace of $\CC H^n$. 
Let $\pr_H(x)$ be the ``projection" of $[x]$ on $H$, that is, the point on $[H]$
that is closest to $[x]$. 
\par
Let $r_1, \dotsb, r_k \in V$ be linearly independent vectors of negative norm and let
$H = r_1^{\bot} \cap \dotsb  \cap r_k^{\bot}$. Assume that $H \cap V^+ \neq \emptyset$, in other
words, the span of $\lbrace r_1, \dotsb, r_k \rbrace$ is negative definite. Then $[H]$ determines
a totally geodesic subspace of $\CC H^n$ and one has
\begin{equation*}
\pr_H(x) = (x + \CC r_1 + \dotsb + \CC r_k) \cap H.
\end{equation*}  
\end{enumerate}
\end{topic}
%
%*******************************************************************************************************
%
\begin{topic}{\bf The setup: }
Let $L$ be the lattice already encountered in the introduction and section \ref{sec-L}
(the direct sum of the Complex Leech lattice and a hyperbolic cell).
Let $Y = \PP_+(L^{\CC}) \simeq \CC H^{13}$ and $\cM$ be the union of the mirrors of the
reflection group of $L$. Let $Y^{\circ} = Y \setminus \cM$ be the complement of the mirrors
and let $X^{\circ} = Y^{\circ}/\Gamma$.
A continuous function $\gamma: [0,1] \to Y^{\circ}$ will be called a path in $Y^{\circ}$.
Let $\Pi_{Y^{\circ}} (a,b)$ be the set of homotopy class of paths in $Y^{\circ}$ beginning
at $a$ and ending at $b$. Given two paths $\gamma$ and $\gamma'$ in $Y^{\circ}$
with $\gamma(1) = \gamma'(0)$, let $\gamma * \gamma'$ be the path obtained
by first following $\gamma$ and then following $\gamma'$, at double the speed.
Let $\gamma_1$ and  $\gamma_2$ be two paths in $Y^{\circ}$ with same beginning and endpoints.
If $\gamma_1$ and $\gamma_2$ are homotopic in $Y^{\circ}$, then we write $\gamma_1 \sim \gamma_2$.
The homotopy class of a path $\gamma$ is denoted by $[\gamma]$.
\end{topic}
\begin{definition}
\label{def-G}
Let $G = \lbrace (\gamma, t) \colon t \in \Gamma, \gamma \in \Pi_{Y^{\circ}}(\br, t \br)\rbrace$.
Then $G$ becomes a group, with multiplication defined by
\begin{equation*}
 (\gamma, t).(\gamma', t') = (\gamma * t \gamma', t t').
\end{equation*}
Define $\pi^G_{\Gamma} : G \to \Gamma$ by $\pi^G_{\Gamma}(\gamma, t) = t$. The kernel of the epimorphism
$\pi^G_{\Gamma}$ is $\pi_1(Y^{\circ}, \br)$. So we have an exact sequence
\begin{equation*}
 1 \to \pi_1(Y^{\circ}) \to G \xrightarrow{\pi^{G}_{\Gamma}} \Gamma \to 1.
\end{equation*}
Let $\tilde{Y}^{\circ} = \cup_{ y \in Y^{\circ}} \Pi_{Y^{\circ}}(\br, y)$ be the universal
cover of $Y^{\circ}$. An element of $\tilde{Y}^{\circ}$ lying above $y \in Y^{\circ}$ is represented by a
path $\lambda: [0,1] \to Y^{\circ}$ such that $\lambda(1) = y$. The group $G$ acts on
$\tilde{Y}^{\circ}$ by
\begin{equation*}
 ([\gamma], t) [\lambda] = [ \gamma * t \lambda].
\end{equation*}
Let $\pi^{\tilde{Y}^{\circ}}_{X^{\circ}} : \tilde{Y}^{\circ} \to X^{\circ}$
be the projection given by $[\lambda] \mapsto \Gamma \lambda(1)$.
Then $\pi^{\tilde{Y}^{\circ}}_{X^{\circ}}(  g [\lambda] )
= \pi^{\tilde{Y}^{\circ}}_{X^{\circ}}([\lambda])$ for all $g \in G$,
that is, $G$ acts as deck transformations on the ramified covering
$\tilde{Y}^{\circ} \to X^{\circ}$.
For the purpose of this article, we define $G$ to be {\it orbifold fundamental group} of 
$X^{\circ}$.
\end{definition}
By the results mentioned in section \ref{sec-L}, there is an surjection
$\phi : \cA(D) \to \Gamma$, taking the generators of $\cA(D)$
to order three simple reflections in $\Gamma$. We want to prove the
following theorem.
\begin{theorem}
\label{th-lift}
There exists a homomorphism $\psi: \cA(D) \to G$ such that $\pi^G_{\Gamma} \circ \psi = \phi$.
\end{theorem}
To give a flavor of the argument, we first prove an easy model lemma.
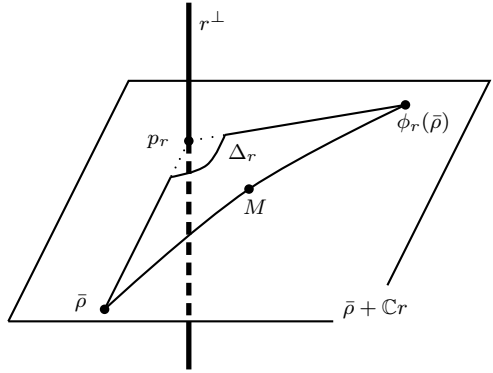
\begin{figure}
\[
\begin{pspicture}(-5,-3)(5,3)
\psset{unit=.8cm}
\psline(-4,-2)(1.4,-2)

\psline(2.3,-1.4)(4,2)(-2,2)(-4,-2) \rput(2.1, -1.8){\tiny $\br + \CC r$}
\rput(-1,1){
\psline[linewidth=.07cm](0,0)(0,2.3) \psline[linewidth=.07cm, linestyle=dashed](0,0)(0,-2.9) \psline[linewidth=.07cm](0,-3)(0,-3.8)
\pscurve(-1.4,-2.8)(1,-.8)(3.6,.6)
\psline(-1.4,-2.8)(-.3,-.6) \psline[linestyle=dotted](-.3,-.6)(0,0) 
\psline(3.6,.6)(.6,.1) \psline[linestyle=dotted](.6,.1)(0,0)
\pscurve(-.3,-.6)(.3,-.4)(.6,.1)
\psdot(0,0) \rput(-.5,0){\tiny $p_r$}
\psdot(-1.4,-2.8) \rput(-1.8,-2.7){\tiny $\br$}
\psdot(3.6,.6) \rput(3.9,.3){\tiny $\phi_r(\br)$}
\rput(.4,2){\tiny $r^{\bot}$}
\psdot(1,-.8) \rput(1.1,-1.1){\tiny $M$} 
\rput(.9,-.2){\tiny $\Delta_r$}
}
\end{pspicture}
\]
\caption{a totally geodesic triangle $\Delta_r$ such that $\Delta_r \setminus \lbrace p_r \rbrace$ does not
intersect any mirror.}
\label{fig-1}
\end{figure}
%
%
%*******************************************************************************************************
%
\begin{lemma}
Let $r$ be a simple root of $L$.
Let $p_r$ be the projection of $\br$ on $r^{\bot}$.
Let $\Delta_r$ be the closed totally geodesic triangle in $Y$ with vertices at
$\br$, $p_r = \op{pr}_{r^{\bot}}(\br)$ and $\phi_r(\br)$
(see figure \ref{fig-1}).
Then $\Delta_r \setminus \lbrace p_r \rbrace$ does not intersect
any mirror. 
\par
(b) Let $\gamma$ and $\gamma'$ be any two paths lying in $\Delta_r \setminus \lbrace p_r \rbrace$,
starting at $\br$ and ending at $\phi_x(\br)$. Then $\gamma \sim \gamma'$ in $Y^{\circ}$.
\label{lemma-def}
\end{lemma}
\begin{proof}
The isosceles triangle $\Delta_r$ is totally geodesic, since it lies in the
complex geodesic $\PP_+(\CC r + \CC \br)$, containing $\br$ and $r$. 
Let $M$ be the midpoint of $\br$ and $\phi_r(\br)$. Let $\Delta_r^1$ be the 
triangle with vertices $\br, p_r, M$ and $ \Delta_r^2$ be the triangle with
vertices $\phi_r(\br), p_r, M$. So $\Delta_r = \Delta_r^1 \cup \Delta_r^2$.
One checks that $d(\br, p_r) > d(\br, M)$. So the point of $\Delta_r^1$
that is furthest from $\br$ is $p_r$. But $d(\br, p_r)$ is
the minimum distance between $\br$ and any mirror (by \ref{th-el}(c) ).
Hence no mirror can intersect $\Delta_r^1 \setminus \lbrace p_r \rbrace$.
The same statement holds for $\Delta_r^2$ by symmetry.
This proves part (a). Part (b) follows from part (a). 
\end{proof}
%
%*******************************************************************************************************
%
\begin{definition}
Maintain the notations of lemma \ref{lemma-def}.
Let $r$ be a simple root. Let $[\gamma]_r$ be the unique homotopy class of paths lying
in $\Delta \setminus \lbrace p_r \rbrace$ starting at $\br$ and ending at $\phi_x(\br)$.
Define $g_r \in G$, called the
{\it braid reflection} in $r$, by $g_r = ([\gamma]_r, \phi_r^{\omega})$.
If $r_1, r_2, \dotsb, r_{26}$ are the simple roots, then we write
$\phi_i = \phi_{r_i}^{\omega}$, $[\gamma]_i = [\gamma]_{r_i}$ and $g_i = g_{r_i}$.
\end{definition}
Now, theorem \ref{th-lift} follows from the following result.
\begin{theorem}
Let $r_1$ and $r_2$ be two simple roots of $L$. 
If the $\omega$-reflections $\phi_1$ and $\phi_2$ braid (resp. commute) in $\Gamma$,
then the braid reflections $g_1$ and $g_2$ braid (resp. commute) in $G$. So $\psi(r_i) = g_i$
defines a homomorphism $\psi: \cA(D) \to G$ such that $\pi^G_{\Gamma} \circ \psi = \phi$.
\label{pi1rel}
\end{theorem}
%
%*******************************************************************************************************
%
The rest of this section is devoted to proving theorem \ref{pi1rel}. 
\begin{proof}[Sketch of proof:] 
We need to set up some notations to make our way smooth.
\par
Let $r_1, \dotsb, r_{26}$ be the simple roots.
Define the vectors $D_{\rho} = \lbrace \rho_1, \dotsb, \rho_{26} \rbrace$ by 
\begin{equation*}
\rho_j = \begin{cases} 
          r_j & \text{\;if \;} r_j \in \cP, \\
	\xi r_j & \text{\; if \;} r_j \in \cL.
         \end{cases}
\end{equation*}
Sometimes it is more convenient to use the vectors $\rho_j$ instead of $r_j$.
Recall that for all $i$ and $j$, $\ip{\rho_i}{\rho_j}$ is a non-negative real number, 
$\br = \sum_{i = 1}^{26} \rho_i/26$ and $\ip{\rho_i}{\br} = \abs{\br}^2$.
\par
Let $r_1$ and $r_2$ be any two distinct simple roots and
$\rho_1$ and $\rho_2$ be the corresponding elements of $D_{\rho}$. Let
\begin{equation}
p_i = \br + \frac{\nr}{3} \rho_i \text{\;\; and \;\;} q = \br + \frac{\nr}{\alpha}(\rho_1 + \rho_2), 
\text{\;\; where\;\;} 
\alpha = 
\begin{cases} 3 & \text{if $\ip{\rho_1}{\rho_2} = 0$}\\
3 - \sqrt{3} & \text{if $\ip{\rho_1}{\rho_2} = \sqrt{3}$} \end{cases}
\label{eq-defalpha}
\end{equation}
Then $[p_1]$, $[p_2]$ and $[q]$ are the projections of
$\br$ on $\rho_1^{\bot}$, $\rho_2^{\bot}$ and $\rho_1^{\bot} \cap \rho_2^{\bot}$  respectively.
\par
Suppose the reflections in $r_1$ and $r_2$ braid. Let $\phi_1 = \phi_{r_1}^{\omega}$ and
$\phi_2 = \phi_{r_2}^{\omega}$.
We define vectors $z_0, \dotsb, z_6$ as follows (see figure \ref{fig-2}):
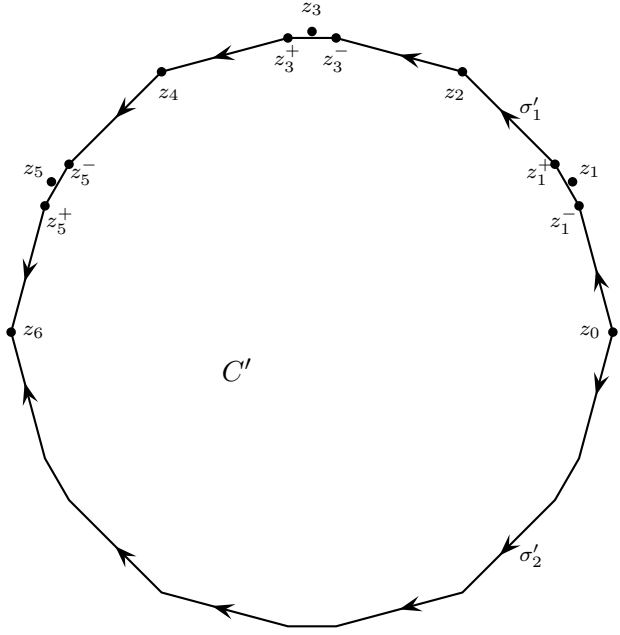
\begin{figure}
\[
\begin{pspicture}(-5,-4.2)(5,4.2)
\psdot(4,0) \psdot(3.550, 1.68) \psdot(3.464,2) \psdot(3.230,2.234)
\psdot(2,3.464)  \psdot(.32,3.914) \psdot(0,4) \psdot(-.32,3.914)
\psdot(-2,3.464) \psdot(-3.230,2.234) \psdot(-3.464,2)  \psdot(-3.550,1.68)
\psdot(-4,0)
\psline(4,0)(3.550, 1.68)(3.230,2.234)(2,3.464)(.32,3.914)(-.32,3.914)(-2,3.464)(-3.230,2.234)(-3.550,1.68)(-4,0)
\psline(4,0)(3.550, -1.68)(3.230,-2.234)(2,-3.464)(.32,-3.914)(-.32,-3.914)(-2,-3.464)(-3.230,-2.234)(-3.550,-1.68)(-4,0)
\psline[arrowsize=6pt]{->}(3.820,.672)(3.775,.84) \psline[arrowsize=6pt]{->}(-3.775,.84)(-3.820,.672) 
\psline[arrowsize=6pt]{->}(2.615,2.849)(2.492,2.972) \psline[arrowsize=6pt]{->}(-2.492,2.972)(-2.615,2.849)
\psline[arrowsize=6pt]{->}(1.328,3.644)(1.16,3.689) \psline[arrowsize=6pt]{->}(-1.16,3.689)(-1.328,3.644) 
\psline[arrowsize=6pt]{->}(3.820,-.672)(3.775,-.84) \psline[arrowsize=6pt]{->}(-3.775,-.84)(-3.820,-.672) 
\psline[arrowsize=6pt]{->}(2.615,-2.849)(2.492,-2.972) \psline[arrowsize=6pt]{->}(-2.492,-2.972)(-2.615,-2.849)
\psline[arrowsize=6pt]{->}(1.328,-3.644)(1.16,-3.689) \psline[arrowsize=6pt]{->}(-1.16,-3.689)(-1.328,-3.644) 
\rput(3.7,0){\tiny $z_0$} \rput(3.350, 1.48){\tiny $z_1^{-}$} \rput(3.694,2.15){\tiny $z_1$} \rput(3.030,2.134){\tiny $z_1^+$}
\rput(1.9,3.164){\tiny $z_2$}  \rput(.325,3.644){\tiny $z_3^-$} \rput(0,4.3){\tiny $z_3$} \rput(-.325,3.644){\tiny $z_3^+$}
\rput(-1.9,3.164){\tiny $z_4$} \rput(-3.030,2.134){\tiny $z_5^-$} \rput(-3.694,2.15){\tiny $z_5$} \rput(-3.350,1.48){\tiny $z_5^+$}
\rput(-3.7,0){\tiny $z_6$}
\rput(2.915,2.979){\tiny $\sigma_1'$} 
\rput(2.915,-2.979){\tiny $\sigma_2'$} 
\rput(-1,-.5){\small $C'$} 
\end{pspicture}
\]
\caption{Schematic picture of the $2$-cell $C'$ and the boundary curves $\sigma_1'$ and $\sigma_2'$.
The points $z_k$ and $z_{2 j - 1}^{\pm}$ on $\sigma_1'$ are marked.}
\label{fig-2}
\end{figure}
\begin{equation*}
z_0 = \br, \text{\;\;} z_1 = p_1, \text{\;\;} z_2 = \phi_1(\br), \text{\;\;} z_3 = \phi_1(p_2), \text{\;\;} 
z_4 = \phi_1 \phi_2 (\br), \text{\;\;} z_5 = \phi_1 \phi_2 (p_1), \text{\;\;} z_6 = \phi_2 \phi_1 \phi_2 (\br). 
\end{equation*}
Note that $\ip{z_{j-1}}{z_j} \in \RR_+$, for $j = 1, \dotsb, 6$. So the
geodesic segment joining $z_{j-1}$ and $z_j$ in $Y$ follows the curve
$\Conv(z_{j-1}, z_j)$. We shall use complex hyperbolic geometry to prove the 
following lemma.
\begin{lemma}
Let $\sigma_1 = \cup_{j = 1}^6 \Conv(z_{j-1},z_j)$. Define $\sigma_2$ similarly by
interchanging the roles of $r_1$ and $r_2$.
Then the $2$-cell $C = \Conv(q, \sigma_1 \cup \sigma_2) \subseteq V_+$ only intersects the
four mirrors $r_1^{\bot}$, $r_2^{\bot}$, $\phi_{r_1}(r_2)^{\bot}$ and
$\phi_{r_2}(r_1)^{\bot}$.
\label{meetdeltabraid}
\end{lemma}
Note that $\sigma_1$ and $\sigma_2$ are curves in $Y$ (not in $Y^{\circ}$).
The curve $\sigma_1$ contains the points $z_1$, $z_3$ and $z_5$ which
are on the mirrors. So we need to modify $\sigma_1$ to
avoid this points. For this, fix a small positive real number $\epsilon$.
For $j = 1, 2, 3$, let
\begin{equation}
z_{2j - 1}^{-} = z_{2j - 1} + \epsilon z_{2j -2} \text{\; and \;} z_{2j - 1}^+ = z_{2j - 1} + \epsilon z_{2j}
\label{eq-defzjpm}
\end{equation}
and let $\tilde{\gamma}_j$ be a path in $V_+$ that goes from $z_{2j -2}$ to $z_{2j}$ along the curve 
\begin{equation}
\Conv( z_{2j - 2}, z_{2j -1}^-) \cup \Conv(z_{2j - 1}^-, z_{ 2j - 1}^+) \cup \Conv(z_{2j-1}^+, z_{2j}).
\label{eq-defgammatilde}
\end{equation}
Define $\sigma_1' = \tilde{\gamma}_1 * \tilde{\gamma}_2 * \tilde{\gamma}_3$ (see figure \ref{fig-2}).
Similarly define $\sigma_2'$, by interchanging the role of $r_1$ and $r_2$.
Lemma \ref{lemma-def} implies that $[\tilde{\gamma}_1] = [\gamma]_1$. Similarly,
$[\tilde{\gamma}_2] = \phi_1[\gamma]_2$ and $[\tilde{\gamma}_3] = \phi_1 \phi_2 [\gamma]_1$.
So $g_1 g_2 g_1 = ([\sigma_1'], t)$. Similarly $g_2 g_1 g_2 = ([\sigma_2' ],t )$.
We need to prove that $\sigma_1' \sim \sigma_2'$. This follows
from the lemma given below.
\begin{lemma}
Let $q' = q + \epsilon \br$. Then the $2$-cell $C' = \Conv( q', \sigma_1' \cup \sigma_2') \subseteq V_+$ does
not intersect any mirrors.
\label{lemma-perturb}
\end{lemma}
Given lemma \ref{meetdeltabraid}, the proof of \ref{lemma-perturb} is easy.
Since $C'$ is a small perturbation of $C$, it suffices to check that
the four mirrors intersecting $C$, do not intersect $C'$. This can
be checked by hand.
This completes the argument in the case when $\phi_1$ and $\phi_2$ braids.
When $\phi_1$ and $\phi_2$ commutes, the argument is similar and the
calculations are simpler. We briefly indicate the modifications 
that are needed. 
\par
Assume that $\phi_1$ and $\phi_2$ commute. Let 
\begin{equation*}
z_0 = \br, \text{\;\;} z_1 = p_1, \text{\;\;} z_2 = \phi_1(\br), \text{\;\;} 
z_3 = \phi_1(p_2), \text{\;\;} z_4 = \phi_1 \phi_2 (\br).
\end{equation*}
Instead of lemma \ref{meetdeltabraid} we have the following lemma:
\begin{lemma}
Let $\sigma_1 = \cup_{j = 1}^4 \Conv(z_{j-1},z_j)$. Define $\sigma_2$ similarly by
interchanging the roles of $r_1$ and $r_2$.
Then the $2$-cell $C = \Conv(q, \sigma_1 \cup \sigma_2)$ only intersects the
mirrors $r_1^{\bot}$ and $r_2^{\bot}$.
\label{meetdeltacommute}
\end{lemma}
Define $z_{2j -1}^{\pm}$ and $\tilde{\gamma}_j$ for $j = 1, 2$, by
the formulas given in \eqref{eq-defzjpm} and \eqref{eq-defgammatilde}. 
Let $\sigma_1' = \tilde{\gamma}_1 * \tilde{\gamma}_2$.
Similarly define $\sigma_2'$. With this setting, one has to re-prove
lemma \ref{lemma-perturb}, which amounts to checking that $r_1^{\bot}$
and $r_2^{\bot}$ do not intersect $C'$.
\end{proof}
%
%
%*******************************************************************************************************
%
It remains to prove the lemmas stated in the sketch above. 
\begin{proof}[proof of \ref{meetdeltabraid} and \ref{meetdeltacommute}].
Let $p_1$, $p_2$ and $q$ be as given in \eqref{eq-defalpha}. 
Note that the inner products between $\br, p_1, p_2$ and $q$ are all real
and they lie in a totally geodesic subspace $P$ of $\CC H^{13}$ isomorphic to the real
hyperbolic plane. 
The $2$-cell $C \subseteq V^+$  given in \ref{meetdeltabraid} (resp. \ref{meetdeltacommute})
is a union of $12$ (resp. $8$) Euclidean triangles and $[C]$ is a union
of $12$ (resp. $8$) totally geodesic triangles in $Y$ (because $q$ has real inner product
with the vertices of $\sigma_1$ and $\sigma_2$). The boundary of $C$ is $\sigma_1 \cup \sigma_2$. 
\par
Consider the quadrilateral $T = \Conv(\br, p_1, q) \cup \Conv(\br, p_2, q)$.
If $\phi_1$ and $\phi_2$ braid, then, with $C$ as given in \ref{meetdeltabraid}, we have,
\begin{equation*}
C = T \cup \phi_1(T) \cup \phi_2(T) \cup 
\phi_1 \phi_2 (T) \cup \phi_2 \phi_1 (T) \cup \phi_1 \phi_2 \phi_1 (T).
\end{equation*}
If $\phi_1$ and $\phi_2$ commute, then, with $C$ as given in \ref{meetdeltacommute}, we have,
\begin{equation*}
C = T \cup \phi_1(T) \cup \phi_2(T) \cup  \phi_1 \phi_2 (T).
\end{equation*}
There is a diagram automorphism that interchanges $\rho_1$ and $\rho_2$, so the intersection of
the triangle $\Conv(\br, p_2, q)$ with the mirrors is exactly similar to that of
the triangle $\Conv(\br, p_1, q)$. 
Lemma \ref{meetdeltabraid} and \ref{meetdeltacommute} now follows from lemma \ref{meetdelta},
given below.
\end{proof}
\begin{lemma} The triangle $\Conv(\br, p_1, q)$ meets the mirror $\rho_1^{\bot}$
along the edge $\Conv(p_1, q)$, meets $\rho_2^{\bot}$ at $q$ and,
in the case when $\ip{\rho_1}{\rho_2} = \sqrt{3}$ it also meets $\phi_{\rho_1}^{\pm}(\rho_2)^{\bot}$
at $q$. Except for these cases no other mirrors intersects
$\Conv(\br, p_1, q) \setminus \lbrace p_1 \rbrace$.
\label{meetdelta}
\end{lemma}
%
%
%*******************************************************************************************************
%
\begin{proof}[proof of lemma \ref{meetdelta}]
The proof given below uses some computer verification. These calculations
were performed using the GP/PARI calculator. The codes are contained
in the file pi1.gp available at www.math.uchicago.edu/\textasciitilde tathagat/codes/index.html.
\par
The diagonal action of $2.\op{PGL}(3, \FF_3)$ on distinct pairs of simple roots has three orbits:
\begin{enumerate}
\item $\rho_1 \in \cP$, $\bar{\xi} \rho_2 \in \cL$ and $\ip{\rho_1}{\rho_2} =\sqrt{3}$;
for calculation we take $\rho_1 = a, \rho_2 = \xi f$.
\item $\rho_1 \in \cP$, $\bar{\xi} \rho_2 \in \cL$ and $\ip{\rho_1}{\rho_2} = 0$; 
for calculation we take $\rho_1 = a, \rho_2 = \xi d_1$.
\item $\rho_1 \in \cP$ and $\rho_2 \in \cP$; 
for calculation we take $\rho_1 = a, \rho_2 = c_1$.
\end{enumerate}
Accordingly, we have to consider three cases in the calculations below.
If not stated otherwise, the statements below are made for all three cases.
It will be convenient to define the {\it height} of a root $r$, denoted by
$\op{ht}(r)$ as follows: 
\begin{equation*}
 \tilde{\op{ht}}(r) = \ip{\br}{r}/\abs{\br}^2 \text{\;\; and \; \;}
\op{ht}(r) = \abs{\tilde{\op{ht}}(r)}.
\end{equation*}
\par
Let $d_0 = d(\br, p_1)$. This is the minimum distance from $\br$ to any mirror.
Suppose $r^{\bot}$ is a mirror that intersects the triangle $\Delta_1 = \Conv(\br, p_1, q)$.
The longest edge of the triangle $\Delta_1$ is $\Conv(\br, q)$. 
So 
\begin{equation*}
\sinh^{-1}\Bigl( \frac{\abs{\ip{r}{\br} }}{\abs{r}\abs{\br} }  \Bigr) = d(r^{\bot}, \br) \leq \md_{\br}(\Delta_1) = d(\br, q).
\end{equation*}
This gives a bound for the possible height of the root $r$:
\begin{equation}
\label{htbound}
\op{ht}(r) = \frac{\abs{\ip{r}{\br}}}{\nr} \leq 
\frac{\abs{r}}{\abs{\br}} \sinh \Bigl( \cosh^{-1}\Bigl(  \frac{\abs{\ip{\br}{q}}}{\abs{\br}\abs{q}} \Bigr) \Bigr)
\leq 2.18.
\end{equation}
Let $s$ be the point on $\Conv(\br, q)$
such that $d(\br, s) = d_0$. The triangle $\Conv(\br, p_1, s)$ cannot meet any
mirror except at $p_1$ and possibly at $s$.
So the mirror $r^{\bot}$ must intersect the triangle $\Delta = \Conv(p_1, q, s)$. The situation is
depicted in the figure \ref{fig-3}.
\begin{figure}
\[
\begin{pspicture}(-5,-2.5)(5,2.5)
\psset{unit=1cm}
\psline[linewidth=.05cm](-2.2,-1.2)(0,1) \rput(-2.4, -1.4){{\tiny $r_1^{\bot}$}}
\psline[linewidth=.05cm]( 2.2,-1.2)(0,1) \rput( 2.4, -1.4){{\tiny $r_2^{\bot}$}}
\psline(0,-2)(0,1) \rput(.2,-2.1){{ \tiny $\br$}} \rput(.2, 1.1){{ \tiny $q$}}
\psline(-1.5,-.5)(0,-2)(1.5,-.5) \rput(-1.7,-.4){{ \tiny $p_1$}} \rput(1.7, -.4){{ \tiny $p_2$}}
\psline(-1.5,-.5)(0,0)(1.5,-.5)  \rput(.2, .15){{\tiny $s$}}
\psline[linewidth=.05cm](-.4,2)(-.4,.4) \rput(-.7,1.8){{\tiny $r^{\bot}$}} 
\psline[linewidth=.05cm, linestyle=dashed](-.4,.4)(-.4,-1.6)
\psline[linewidth=.05cm](-.4,-1.6)(-.4,-2.2) 
\end{pspicture}
\]
\caption{}
\label{fig-3}
\end{figure}
One can take $s = \br + c( \rho_1 + \rho_2)$ where $c \in \RR_+$ is a constant to be determined. 
The equality $d(s, \br) = d_0$ implies
\begin{equation*}
1 + \frac{\nr}{3} 
= \frac{\abs{\ip{\br}{p_1}}^2}{\nr \abs{p_1}^2} 
= \frac{\abs{\ip{s}{\br}}^2}{\abs{s}^2 \nr } 
= \frac{\abs{\br}^4 ( 1 + 2c)^2}{ (\nr + 4 c \nr - 2 \alpha c^2 )\nr} 
\end{equation*}
(where $\alpha$ is given in \eqref{eq-defalpha}). Rearranging, one has the quadratic equation
\begin{equation*}
( 4 + \tfrac{2 \alpha}{\nr}( 1 + \tfrac{\nr}{3}) ) c^2 - \tfrac{4}{3}\nr c  - \tfrac{\nr}{3} = 0
\end{equation*}
with one positive and one negative root. The positive root gives the required $c$.
\par
Divide the triangle $\Delta$ into two triangles $\Delta'$ and $\Delta''$ by
joining $s$ with the midpoint of $\Conv(p_1,q)$. For $i = 1, \dotsb, 26$, one has,
\begin{equation*}
d(\rho_i^{\bot}, r^{\bot}) \leq \md_{\rho_i^{\bot}}(\Delta) 
= \max\lbrace \md_{\rho_i^{\bot}}(\Delta'), \md_{\rho_i^{\bot}}(\Delta'') \rbrace
\end{equation*}
where the equality is an instance of equation \eqref{mdunion} in \ref{distbasic}.
We estimate $\md_{\rho_i^{\bot}}(\Delta')$ and $\md_{\rho_i^{\bot}}(\Delta'')$
using the inequality \eqref{mdHD} of \ref{distbasic}. This gives a bound on
$\abs{\ip{\rho_i}{r}}^2 = 9 \cosh^2( d(\rho_i^{\bot}, r^{\bot}))$.
By explicit computation, this bound is strictly less than 12 except in three cases
when $\ip{\rho_1}{\rho_2} = \sqrt{3}$ and $\bar{\xi} \rho_i \in \cL$. 
Since the inner product between two vectors of $L$ always lies in
$p \cE$ one has
\begin{equation}
\label{ipxbound}
\abs{\ip{x_i}{r}}^2 \in \lbrace 0, 3, 9 \rbrace \text{\;\; for\;\;} i = 1, \dotsb, 13.
\end{equation} 
Similarly, from the inequality
\begin{equation*}
d(w_{\cP}, r^{\bot}) \leq \md_{w_{\cP}}(\Delta) 
\leq \max \lbrace d(w_{\cP}, p_1), d(w_{\cP}, q), d(w_{\cP}, s) \rbrace,
\end{equation*}
one gets a bound on $\abs{\ip{w_p}{r}}^2 = 9 \sinh^2( d(r^{\bot}, w_{\cP}))$.
By explicit computation the bound is strictly less than $10$. So one has
\begin{equation}
\label{ipwpbound}
\abs{\ip{w_{\cP}}{r}}^2 \in \lbrace 0, 3, 9 \rbrace.
\end{equation}
\par
The conditions \eqref{ipxbound} and \eqref{ipwpbound} restrict the possibilities for $r$ to
a finite set as in the proof of proposition 6.1 in \cite{TB:EL}.
Write $r$ in terms of the orthogonal basis $(w_{\cP}, x_1, \dotsb, x_{13})$ as
\begin{equation*}
r = \ip{w_{\cP}}{r} w_{\cP}/3 - \sum \ip{x_i}{r} x_i/3 .
\end{equation*}
Taking norm and re-arranging one gets 
\begin{equation*}
\sum \abs{\ip{x_i }{r} }^2 = 9 + \abs{\ip{w_{\cP}}{r} }^2
\end{equation*}
which leaves only finitely many possibilities for the vector $(\ip{w_{\cP}}{r}, \ip{x_1}{r}, \dotsb, \ip{x_{13}}{r})$.
The possible inner products of $r$ with $x_i$ and $w_p$ are shown in table \ref{table-1}.
\begin{table}
\begin{tabular}{|c|c|c|c|c|}
\hline
& & & & \\
$j$ &  $\ip{w_{\cP}}{r}$ & possible $(\ip{x_1}{r}, \dotsb, \ip{x_{13}}{r} )$ & $\tilde{\op{ht}}(r)$ 
& $\min(\op{ht}(r))$ \\
 & & & & \\
\hline
$1$ & $0$ & $(3u_1, 0^{12})$ & $-u_1$ & $1$  \\
  & & & & \\
$2$ & & $\theta(u_1, u_2, u_3, 0^{10})$ & $\tfrac{1}{\theta}\sum_{i=1}^3 u_i$ & no root\\
& & & &\\
\hline
$3$ & $\theta$ & $(3 u_1, \theta u_2, 0^{11})$ &$\tfrac{1}{\theta}(-4 - \sqrt{3}+u_1 -\theta u_2)$ & no root\\
& & & & \\
$4$ &    & $\theta(u_1, u_2, u_3, u_4, 0^{9})$ &$\tfrac{1}{\theta}(-4 - \sqrt{3}+\sum_{i=1}^4 u_i )$ & $1$ \\
& & & & \\
\hline
$5$ & $3$    & $3(u_1, u_2, 0^{11})$ & $ 4  +\sqrt{3} - u_1 - u_2$ & $2 + \sqrt{3}$\\ 
& & & &\\
$6$ &    & $(3u_1,\theta u_2,\theta u_3,\theta u_4,0^9)$ & $ 4+\sqrt{3}-u_1+
\tfrac{1}{\theta}\sum_{i=2}^4 u_i $  & $3.24$\\
& & & & \\
$7$ &    & $\theta(u_1,\dotsb, u_6, 0^7)$ & $4 + \sqrt{3} +\tfrac{1}{\theta}\sum_{i=1}^6 u_i$ & $2.73$\\
& & & & \\
\hline
%$8$ & $2 \theta$ & $(3 u_1, 3 u_2, \theta u_3, 0^{10})$ & $\tfrac{i}{\nr}-u_1-u_2 +\tfrac{1}{\theta} u_3$ & $4.309 $\\ 
%& & & &  \\          
%$9$ &           & $(3u_1, \theta u_2, \theta u_3, \theta u_4, \theta u_5, 0^8)$ &$\tfrac{i}{\nr} - u_1 + \tfrac{1}{\theta}\sum_{i=2}^5 u_i$ & $3.479$\\
%& & & & \\       
%$10$ &       & $\theta(u_1, \dotsb, u_7, 0^6)$ & $\tfrac{i}{\nr} + \tfrac{1}{\theta}\sum_{i=1}^7 u_i$ & $2.577$\\
%& & & & \\
%\hline
\end{tabular}
\caption{In each row, we consider roots $r$ of a certain form, determined by the
possible inner products of $r$ with $w_{\cP}$ and with $x_1, \dotsb, x_{13}$ as shown 
in the second and third column respectively. ($u_1, u_2, \dotsb$ stand for sixth roots
of unity.) The roots considered in the $j$-th row are called the roots of type $j$.
The fourth column records $\tilde{\op{ht}}(r) = \ip{\br}{r}/\nr$ for a root $r$ of type $j$.
The fifth column records the minimum possible absolute value of the entry in the fourth
column as $u_i$'s vary over the sixth roots of unity, that is, the minimum possible
height of a root of type $j$. }
\label{table-1}
\end{table}
By inspection of the last three entries of fifth column, we find that the minimum height of a root
having type 5, 6 or 7, is greater than the bound on $\op{ht}(r)$ obtained in \eqref{htbound}.
So mirrors of type 5, 6 or 7 cannot intersect $\Delta_1$. Next, one checks that there are no root
of type 2 or 3. This implies $r$ must be of type 1 or 4. The only mirrors of type 1 are the
thirteen simple mirrors corresponding to $\cP$. Now we make a list of those
type 4 roots that satisfy $\op{ht}(r) < 2.2$ (which is enough in view of inequality \eqref{htbound}).
The list consists of the thirteen simple mirrors corresponding to $\cL$ and the
$104$ mirrors corresponding to the roots of the form $\phi_x^{\pm}(l)$,
where $x \in \cP$, $y \in \cL$ and $x$ is incident on $l$ 
(Experimentally, these are the only roots having height  $\abs{1 + \xi}$.)
We want to show that these $130$ mirrors do not meet $\Delta_1$ except for the cases
described in the statement of lemma \ref{meetdelta}. This is checked on the computer as
follows:
\par
Let $P = \lbrace [\br + s_2 p_1 + s_3 q] \colon s_j \in \RR \rbrace \subseteq \PP(V)$. 
Then $P$ contains the triangle $\Delta_1$.
If the complex numbers $\ip{r}{\br}$, $\ip{r}{p_1}$ and $\ip{r}{q}$ all have the same
argument then we find that $\op{Re}(\ip{r}{\br}) = \op{Re}(\ip{r}{p_1}) = \op{Re}(\ip{r}{q}) = \nr$.
So a convex combination of $\br$, $p_1$ and $q$ cannot be orthogonal to $r$.
Otherwise $r^{\bot} \cap P$ is a point in $\PP(V)$ which can be found by
solving two linear equations. We have checked that this point does not belong to
$\Delta_1$ except in the cases mentioned in the statement of the lemma.
\end{proof}
%
%
%*******************************************************************************************************
%
%
\begin{proof}[proof of lemma \ref{lemma-perturb}]
Assume that $r_1 \in \cP$, $r_2 \in \cL$ and that $\phi_{r_1}$ and $\phi_{r_2}$ braid. Let 
\begin{equation*}
r_3 = \phi_{r_1}(r_2) = \phi_{r_2}^{-1} (r_1) = r_2 + r_1  \text{\; and \;} 
 r_4 = \phi_{r_2}(r_1) = \phi_{r_1}^{-1} ( - \omega r_2) = r_1 - \omega r_2.
\end{equation*}
It suffices to show that the mirrors $\lbrace r_1^{\bot}, \dotsb, r_4^{\bot} \rbrace$
do not meet $C'$.
For this, it is enough to show that $\Conv(q', \sigma_1')$ do not intersect
these four mirrors (since the same argument applies to $\Conv(q', \sigma_2')$ by
symmetry). For $k = 1, \dotsb, 6$, let $\Delta'_k = \op{Conv}(q', z_{k-1}, z_k)$.
For $k = 1, 2, 3$, let $\Delta''_k = \op{Conv}(q', z_{2 k-1}^{-} , z_{2k -1}^+ )$.
So
\begin{equation*}
\op{Conv}(q', \sigma_1') \subseteq  \bigl( \cup_{k = 1}^3 \Delta''_k \bigr) \cup \bigl(  \cup_{k=1}^6 \Delta'_k 
\setminus \lbrace z_1, z_3, z_5 \rbrace \bigr).
\end{equation*}
In table \ref{table-2}, we have recorded some inner products that we are going to need.
\begin{table}
\begin{tabular}{c|ccccccccccccccc|c}
$j \setminus k$ && $0$ && $1$ && $2$ && $3$  && $4$ && $5$ && $6$ && $r_j$\\
\hline 
$1 \;\;$ && $1$         && $0$         &&  $\bar{\omega} $ &&  $ \bar{\omega} c $     &&  $\bar{\omega} - i$  && $-i c$     && $-i$  && $r_1$\\
$2 \;\;$ &&  $\bar{\xi}$ && $\bar{\xi}c$ && $\bar{\xi}-\bar{\omega}$ &&  $-\bar{\omega} c $  && $-\bar{\omega}$ && $0$  && $-\omega$ &&$r_2$\\
$3 \;\;$ &&  $1 + \bar{\xi}$ && $\bar{\xi} c$ && $\bar{\xi}$ && $0$ && $-i$ && $-i c$ && $-i - \omega$ && $r_1 + r_2$\\
$4 \;\;$ && $1 + \xi$ && $\xi c$ && $\xi - \omega$ && $-\omega c$ && $-\omega - i$ && $-i c$  && $-i + \bar{\omega}$ && $r_1 - \omega r_2$\\
\hline
$z_k$ && $\br$ && $p_1$ && $\phi_x(\br)$      && $\phi_x(p_2)$ && $\phi_x \phi_l(\br)$ && $\phi_x \phi_l(p_1)$ && $\phi_l \phi_x \phi_l (\br)$ && \\
&& && && && && && &&
\end{tabular}
\caption{The left and top margin gives the row and column numbers (denoted by $j$ and $k$). 
The $j$-th row of right margin records $r_j$. The $k$-th column in the
bottom margin records $z_k$. 
The entry in the $j$-th row and $k$-th column is $c_j^k = \ip{z_k}{r_j}/\abs{\br}^2$.
Finally, $c = 1 + 3^{-1/2}$.}
\label{table-2}
\end{table}
Looking at the table we make the following observations:
\begin{enumerate}
\item The numbers $c_j^0$ are always nonzero. In each row, atmost one entry is zero.
\item For all $j$ and $k$, we have $\op{Re}( \xi c_j^k ) \geq 0$.   
\item For each $k$, there is an open half plane $P_k \subseteq \CC$
such that  $\lbrace c_j^0, \dotsb, c_j^6 \rbrace \subseteq P_k \cup \lbrace 0 \rbrace$. 
\end{enumerate}
Let $s_1, s_2, s_3 \in \RR$ such that $s_i \geq 0$ and $s_1 + s_2 + s_3 = 1$. 
\par
Let $k \in \lbrace 1, \dotsb, 6\rbrace$ and let $w_k = s_1 q' + s_2 z_{k-1} + s_3 z_k \in \Delta'_k$. 
If $\ip{w_k}{z_j} = 0$, then
\begin{equation}
\ip{w_k}{r_j}/\nr = s_1 \epsilon c_j^0 + s_2 c_j^{k-1} + s_3 c_j^k = 0, 
\label{eq-ipwr}
\end{equation}
where $c_j^k$ are given in table \ref{table-2}. 
The three numbers $\lbrace \epsilon c_j^0, c_j^{k-1}, c_j^k \rbrace$  belong to
$P_k \cup \lbrace 0 \rbrace$ for some open half plane $P_k$.
So a convex combination of these three numbers is zero only when
$c_j^{k-1} = 0$ and $(s_1, s_2, s_3) = (0,1,0)$ or $c_j^k = 0$ and $(s_1, s_2, s_3) = (0,0,1)$.
It follows that equation \eqref{eq-ipwr} holds if and only if the pair $(w_k, r_j)$ is equal to
$(z_1, r_1)$, or $(z_3, r_3)$, or $(z_5, r_2)$.
So the only intersection of $\cup_{k = 0}^6 \Delta'_k $ with $\cup_{j=1}^4 r_j^{\bot}$
is at the points $\lbrace z_1, z_3, z_5 \rbrace$.
\par
Now let $k \in \lbrace 1, 2, 3 \rbrace$ and let 
$w_k = s_1 q' + s_2 z_{2k - 1}^{-} + s_3 z_{2k -1}^+ \in  \Delta''_k$.
Then 
\begin{equation}
\ip{w_k}{r_j}/\nr = s_1 \epsilon c_j^0 + s_2 (c_j^{2 k-1} + \epsilon c_j^{2k -2})  + s_3 (c_j^{2 k - 1} + \epsilon c_j^{2k}). 
\end{equation}
Note that the three numbers
$\lbrace \epsilon c_j^0 , c_j^{2 k-1} + \epsilon c_j^{2k -2} , c_j^{2 k - 1} + \epsilon c_j^{2k} \rbrace$
always belong to some open half plane $P_k$. So a convex combination of these cannot be zero, 
that is, $\cup_{k = 1}^3 \Delta''_k$ does not intersect $\cup_{j = 1}^4 r_j^{\bot}$.
When $\phi_1$ and $\phi_2$ commutes, the calculations are much easier and are omitted.
\end{proof}
%
%
%*******************************************************************************************************
%
%
%
%*******************************************************************************************************
%
%
%*******************************************************************************************************
%
\section{The fixed points of diagram automorphisms}
%
%*******************************************************************************************************
%
%
%*******************************************************************************************************
%
\label{sec-fixed}
Recall, from section \ref{sec-L}, that the diagram automorphisms $\op{PGL}(3, \FF_3)$ point-wise fix a two
dimensional primitive sub-lattice $F$ spanned by $w_{\cP}$ and $w_{\cL}$. Let
$z_0 = w_{\cP} + \theta w_{\cL}$. Let $\Gamma_F = \lbrace g \in \PP \Aut(L) \colon g(F) = F \rbrace$. 
\begin{theorem}
(a) There is an isometry $\beta: \PP_+(F^{\CC}) \to \cH^2$ such that $\beta(\br) = i$,
$\beta(w_{\cP}) = p$, $ \beta(w_{\cL}) = -p^{-1}$ and $\beta(z_0) = 0$. (See figure \ref{fig1}).
\newline
(b) The isometry $\beta$ induces a group homomorphism $c_\beta: \Gamma_F \to \op{PSL}_2(\RR)$,
given by $c_{\beta}(h) = \beta \circ h \circ \beta^{-1}$. 
The image of $c_{\beta}$ contains the congruence subgroup $\Gamma(13) \subseteq \op{PSL}_2(\ZZ)$.
In-fact 
\begin{equation*}
c_{\beta}(\Gamma_F) \cap \op{PSL}_2(\ZZ) = \nu^{-1} \Gamma_0(13) \nu  
\end{equation*}
where $\nu = \bigl( \begin{smallmatrix} 0 & 1 \\ -1 & 5 \end{smallmatrix} \bigr)$. Further, 
$c_{\beta}(\sigma) = S =  \bigl( \begin{smallmatrix} 0 & -1 \\ 1 & 0 \end{smallmatrix} \bigr) $,
where $\sigma \in \op{Aut}(F)$ is the involution that corresponds to interchanging the
points and lines of $\PP^2(\FF_3)$ .
\label{FisoH}
\end{theorem}
\begin{remark}
The stabilizer in $\Aut(L)$ of the norm zero vector $z_0$ contains $\op{PGL}(3, \FF_3)$.
Theorem 4 of \cite{DA:Y555} implies that $z_0$  is a ``cusp of Leech type", that is,
$z_0^{\bot}/z_0$ is isomorphic to the complex Leech lattice. 
\end{remark}
\begin{figure}
 \[
  \begin{pspicture}(-7,-.6)(7,2.5)
\psset{unit=1.6cm}
\psline[linestyle=dotted](-2.5,0)(3.5,0)
\psarc(0,0){1}{0}{180}
\rput(1,0){\psarc(0,0){1}{0}{180}}
\rput(2,0){\psarc(0,0){1}{0}{180}}
\rput(-1,0){\psarc(0,0){1}{0}{180}}
\psline(-.5,1.7)(-.5,.287)
\psline(.5,1.7)(.5,.287)
\rput(-.333,0){\psarc(0,0){.333}{0}{120}}
\rput(.333,0){\psarc(0,0){.333}{60}{180}}
\pscircle[fillstyle=solid, fillcolor=white](0,1){.06}
\rput(0,1.2){\tiny $\beta(\br)$} 
\pscircle[fillstyle=solid, fillcolor=white](0,0){.06} 
\rput(0,-.2){\tiny $\beta(z_0)$} 
\pscircle[fillstyle=solid, fillcolor=white](1.5,.866){.06}
\rput(1.5,1.1){\tiny $\beta(w_{\cP})$} 
\pscircle[fillstyle=solid, fillcolor=white](-.5,.288){.06} 
\rput(-.8,.2){\tiny $\beta(w_{\cL})$} 
\end{pspicture}
\]
\caption{}
\label{fig1}
\end{figure}
\begin{remark} 
\label{remark-aut}
Let $\Phi_L$ be the set of roots of $L$ and $L^{\CC}_+$ denote the set of vectors of $L^{\CC}$
having positive norm. Consider $E_m: L^{\CC}_+ \to \CC$, defined by,
\begin{equation*}
E_m(z) = \sum_{ r \in \Phi_L} \ip{r}{z}^{-6 m}.
\end{equation*}
Fix $z \in L^{\CC}_+$. The number of roots $r$ such that $\abs{\ip{r}{z}} \leq N$ grows as a
polynomial in $N$. So the infinite sum above defines a non-constant meromorphic functions if
$m$ is large enough. The functions $E_m$ are invariant under $\Aut(L)$ and have poles exactly
along the mirrors of the reflection group $R(L)$. The lattice $F$ is not contained in any
mirror as there are no mirrors passing through $\br \in F^{\CC}$. So the restriction of
$E_m$ to $F^{\CC}_+$ is a non-constant meromorphic function invariant under $\Gamma_F$. 
Theorem \ref{FisoH} shows that $\Gamma_F$ is commensurable with $\op{PSL}_2(\ZZ)$.
So the restriction of $E_m$ to $F^{\CC}_+$ is a ordinary meromorphic modular form of level $13$.
\par
Starting with modular forms of singular weight with poles at cusps, one can construct meromorphic
automorphic forms of type $U(1,n)$ by first taking Borcherds singular theta lift to get an automorphic form of
type $O(2,2n)$ and then restricting to the hermitian symmetric space of $U(1,n)$ (see theorem 7.1 of
\cite{DA:Leech} and \cite{REB:Aut}).
In our example, these Borcherds forms will be automorphic with respect to some finite index subgroup
of $\Aut(L)$ and have their divisors along the mirrors, just like the functions
$E_m(z)$. It will be interesting to understand how the restriction to $F^{\CC}_+$
is related to the inverse of the Borcherds lift.
\end{remark}
It follows from general considerations that the image of $\beta$ is
commensurable with $\op{PSL}_2(\ZZ)$. (This was explained to me by Stephen Kudla.)
But we want to calculate the precise image. So We need to understand when an automorphism
of $F$ extend to an automorphism of $L$. For this, we need the following lemma. We urge the
reader to recall, at this point, the construction of $L$ from the diagram $D$ given in section
\ref{sec-L} and equations \eqref{ipd}, \eqref{defwpwl} and \eqref{eq-ipwpwl}.
\begin{lemma}
Fix any $x_0 \in \cP$ and $l_0 \in \cL$. Let $\bar{L} = L/(F \oplus F^{\bot})$.
Let 
\begin{equation*}
13 \bar{x} = \Sigma_{\cP} = \sum_{x \in \cP} x.
\end{equation*}
Let $\pi_F: L^{\CC} \to F^{\CC}$ and $\pi_{F^{\bot}} : L^{\CC} \to (F^{\bot})^{\CC}$ be the orthogonal projections.
\par
(a) The lattice $F^{\bot}$ is spanned by the vectors 
$\lbrace (x - x_0) \colon x \in \cP \rbrace \cup \lbrace (l - l_0) : l \in \cL \rbrace$.
\par
(b) One has $\Sigma_{\cP} = 4 w_{\cP} - \bar{p} w_{\cL} \in F$
and $(13 x_0 - \Sigma_{\cP}) \in F^{\bot}$. So $\pi_F(x_0) = \bar{x}$ and $\pi_{F^{\bot}}(x_0) = x_0 - \bar{x}$.  
\par
(c) One has an isomorphism $ \cE/13 \cE \simeq \bar{L}$ obtained by sending 1 to the image of $x_0$ (or $l_0$). 
\par
(d) Let $i_F : \bar{L} \to \op{D}(F) = F'/F$ be the injection given by $u \mapsto \pi_F(u) \bmod F$.
Similarly define $i_{F^{\bot}} : \bar{L} \hookrightarrow \op{D}(F^{\bot})$.
It follows from part (b) and (c) that $i_F(\bar{L})$
is generated by $\bar{x}$ and $i_{F^{\bot}}(\bar{L})$ is generated by $(x_0 - \bar{x})$.  
The quotients $\op{D}(F)/i_F(\bar{L})$ and $\op{D}(F^{\bot})/i_{F^{\bot}}(\bar{L})$
are $3$-groups. 
\par
(e) Given $g_1 \in \Aut(F)$ and $g_2 \in \Aut(F^{\bot})$, one can extend $(g_1 \times g_2)$ to an
automorphism of $L$ if and only if
\begin{equation*}
 i_F^{-1} \circ g_1 \circ i_F = i_{F^{\bot}}^{-1} \circ g_2 \circ i_{F^{\bot}}.
 \end{equation*}
(In other words, $g_1$ and $g_2$ acts on the image of $\bar{L}$ in the
same way).
\par
(f) The automorphism $\sigma$ acts on $i_{F^{\bot}}(\bar{L}) \subseteq \op{D}(F^{\bot})$
as multiplication by $3\bar{p}$.
\label{FFp}
\end{lemma}
\begin{proof}
(a) Let $\tilde{F} = \op{span}\lbrace x - x_0, l - l_0 \colon x \in \cP, l \in \cL \rbrace$.
Equation \eqref{eq-ipwpwl} implies $\tilde{F} \subseteq F^{\bot}$. Since $L$ is spanned by $\cP \cup \cL$, it is also
spanned by $\tilde{F}$ together with $x_0$ and  $l_0$. Given $w \in F^{\bot}$, we can write
$w = \alpha_1 x_0 + \alpha_2 l_0 + \tilde{w}$, where $\tilde{w} \in \tilde{F}$. Taking inner product
with $w_{\cP}$ and $w_{\cL}$, one finds that $\alpha_1 = \alpha_2 = 0$, so $w \in \tilde{F}$.
This proves part (a).
\par
(b) For each of the four lines $l$ passing through $x_0$, we have
$w_{\cP} = \bar{p} l + \sum_{x \in l} x$ (cf. equation \eqref{defwpwl}).
Adding these together we get
\begin{equation*}
4 w_{\cP} = \bar{p} \sum_{ x_0 \in l} l + 3 x_0 + \Sigma_{\cP} =  \bar{p} w_{\cL} + \Sigma_{\cP}.
\end{equation*}
So $\Sigma_{\cP} \in F$. Clearly $13 x_0 - \Sigma_{\cP} = \sum_{x \in \cP}(x_0 - x) \in  F^{\bot}$.
\par
(c) Note that $\bar{p} l_0 + 4 x_0 = w_{\cP} + \sum_{ x \in l_0} (x_0 - x)  \in F \oplus F^{\bot}$.
Similarly $p x_0 + 4 l_0 \in F \oplus F^{\bot}$. 
So 
\begin{equation}
l_0 - 3p x_0 = (p x_0 + 4 l_0) - p(\bar{p} l_0 + 4 x_0) \in F \oplus F^{\bot}.
\label{inFpFp}
\end{equation}
So $F \oplus F^{\bot}$ together with $x_0$ generate $L$.
Part (b) implies that $13 x_0 \in F \oplus F^{\bot}$.
Conversely suppose $\alpha x_0 \in F \oplus F^{\bot}$. Choose $u, v \in \cE$ such that 
$(\alpha x_0 - u w_{\cP} - v w_{\cL}) \in F^{\bot}$. Taking inner product with $w_{\cL}$ and $w_{\cP}$ yield
$\alpha \bar{p} = 4 \bar{p} u + 3 v$ and $ 0 = 3 u + 4 p v$. So
\begin{equation*}
-13pv = p( 4 \bar{p} u + 3 v) - 4(3 u + 4 p v) = 3 \alpha ,
\end{equation*}
that is, $\alpha$ is a multiple of $13$.
\par
(d) By calculating discriminants we find that $\abs{\op{D}(F)} = 13^2.3^2$ and $\abs{\op{D}(F^{\bot})} = 13^2.3^{12}$.
Since $\abs{\bar{L}} = 13^2$, it follows that
$\op{D}(F)/i_F(\bar{L})$ and $\op{D}(F^{\bot})/i_{F^{\bot}}(\bar{L})$ are $3$-groups. 
\par
(e) Note that $i_F(\bar{L})$ is generated by $\pi_F(x_0) = \bar{x}$
and $i_{F^{\bot}}(\bar{L})$ is generated by $\pi_{F^{\bot}}(x_0) = x_0 - \bar{x}$.
From part (d) and Chinese remainder theorem we find that $\op{D}(F)$ is the direct product
of $i_F(\bar{L})$ and a $3$-group. 
So $g_1$ acts on $i_F(\bar{L})$ as multiplication by some $\lambda \in (\cE/13\cE)^*$
(since $\bar{L}$ is a cyclic $\cE$-module and $g_1$ is $\cE$-linear). 
Similarly, as $\op{D}(F^{\bot})$ is the product of $i_{F^{\bot}}(\bar{L})$ and a $3$-group,
$g_2$ acts on $i_{F^{\bot}}(\bar{L})$ as multiplication by some $\lambda'$. 
The condition $ i_F^{-1} \circ g_1 \circ i_F = i_{F^{\bot}}^{-1} \circ g_2 \circ i_{F^{\bot}}$
is equivalent to $\lambda = \lambda'$.
If this condition is satisfied, we have 
\begin{equation*}
(g_1 \times g_2)(x_0) = g_1( \bar{x} ) + g_2( x_0 - \bar{x})
\equiv \lambda \bar{x} + \lambda (x_0 - \bar{x}) \bmod F \oplus F^{\bot} 
\equiv \lambda x_0 \bmod F \oplus F^{\bot} ,
\end{equation*}
which implies that $(g_1 \times g_2)(x_0) \in L$. The converse is also clear.
\par
(f) Assume that $\sigma(l _0) = x_0$.
Let $\bar{l} = \sum_{ l \in {\cL}} l/13$. Then $\pi_F(l_0) = \bar{l}$.
We have seen in equation \eqref{inFpFp} that $v = l_0 - 3p x_0 \in F \oplus F^{\bot}$.
So 
\begin{equation*}
v - \pi_F(v) = (l_0 - 3 p x_0 ) - (\bar{l} -3p \bar{x} ) \in F^{\bot},
\end{equation*}
that is, $(l_0 - \bar{l}) \equiv 3p (x_0 - \bar{x}) \bmod F^{\bot}$. It follows that
\begin{equation*}
\sigma(x_0 - \bar{x}) = -\omega( l_0 - \bar{l}) \equiv  -\omega (3p)( x_0 - \bar{x}) \bmod F^{\bot}.
\end{equation*}
\end{proof}
\begin{definition}[of the isometry $\beta$]
\label{npm}
Note that $\ip{w_{\cP}}{\xi w_{\cL}} = 4 \sqrt{3} \in \RR$. We let
\begin{equation*} 
y_{\pm} = 2^{-\frac{1}{2}} 3^{-\frac{1}{4}} (p \pm i)^{-1} (w_{\cP} \pm \xi w_{\cL}).
\end{equation*}
 Then $\lbrace y_+, y_- \rbrace$ forms a basis for the vector space $F^{\CC}$
 and we have 
 \begin{equation*} 
 \abs{y_+}^2 = - \abs{y_-}^2 = 1, \;\;\;\; \ip{y_+}{y_-} = 0.
  \end{equation*}
So $\beta' : \PP_+(F^\CC) \to B^1(\CC)$ given by $\beta'[ u y_+ + v y_-] = \frac{v}{u}$
is an isometry. 
Let $C(w) =i\bigl( \tfrac{ 1 + w}{1  - w} \bigr)$ be the Cayley isomorphism from the disc to the upper half plane. Define
\begin{equation*} 
\beta = C \circ \beta' : P_+(F^\CC) \to \mathcal{H}^2.
 \end{equation*}
 The map $\beta$ is an isometry since both $\beta'$ and $C$ are.
After some simplification, one obtains
\begin{equation}
\beta[a w_{\cP} + b w_{\cL} ] = \frac{ p a + \omega b}{ a + \bar{p} b}
\text{\;\; and \;\;} 
\beta^{-1}(\tau) = [ ( 1 + p \tau) w_{\cP} + \omega^2( \tau - p) w_{\cL} ].
\label{defbeta}
\end{equation}
{\bf Notation: } For the rest of this section, we shall represent $ z= u w_{\cP} + v w_{\cL} \in F^{\CC}$
as a column  vector $z = \begin{pmatrix} u \\ v \end{pmatrix}$.  
Accordingly, the isometry $\beta$ will be represented by the matrix
$\begin{pmatrix} p & \omega \\ 1  & \bar{p} \end{pmatrix}$.
\end{definition}
\begin{lemma}
(a) Let $J_F 
=  \begin{pmatrix} \ip{w_{\cP}}{w_{\cP}} & \ip{w_{\cP}}{w_{\cL}}
\\ \ip{w_{\cL}}{w_{\cP}} & \ip{w_{\cL}}{w_{\cL}} \end{pmatrix}
 = \begin{pmatrix} 3 & 4p  \\ 4 \bar{p} & 3 \end{pmatrix} $
(see equation \eqref{eq-ipwpwl}). Then
\begin{equation}
J_F = \bar{\theta} \beta^* S \beta.
\label{jbeta}
\end{equation}
Let $\beta z = \bigl( \begin{smallmatrix} \tau_1 \\ \tau_2 \end{smallmatrix} \bigr) $,
 $\beta z' = \bigl( \begin{smallmatrix} \tau_1' \\ \tau_2' \end{smallmatrix} \bigr) $.
Then $\ip{z}{z'} = \bar{\theta}( \tau_1' \bar{\tau}_2 - \tau_2' \bar{\tau}_1)$.
In particular $\abs{z}^2 = 2 \sqrt{3} \op{Im}(\tau_1  \bar{\tau}_2)$.
\newline
(b) Let $\nu = \bigl(\begin{smallmatrix} 0 & 1 \\ -1 & 5 \end{smallmatrix}\bigr)$.
Given $g =\bigl( \begin{smallmatrix} a & b \\ c  & d \end{smallmatrix} \bigr) \in
\op{SL}_2(\ZZ)$, let $g_1= \beta^{-1} g \beta$.
 Then $(z \mapsto g_1 z)$ is an isometry of the Hermitian vector space $F^{\CC}$.
Further, $g_1 \in \Aut(F)$ if and only if 
\begin{equation*}
g \in \nu^{-1} \Gamma_0(13) \nu 
= \Big\lbrace \begin{pmatrix} a & b \\ c & d \end{pmatrix} \in \op{SL}_2(\ZZ) \colon 
5(a - d) + (b + c) \equiv 0 \bmod 13  \Big\rbrace. 
\end{equation*}
\label{betalemma}
\end{lemma}
\begin{proof}
The equation \eqref{jbeta} is verified by multiplying matrices.
Given $z , z' \in F^{\CC}$, one has 
\begin{equation*}
\ip{z}{z'} = z^* J_F z' = \bar{\theta} z^* \beta^* S \beta z'
= \bar{\theta} (\bar{\tau}_1  \; \; \bar{\tau}_2 ) S 
\bigl( \begin{smallmatrix} \tau_1' \\ \tau_2' \end{smallmatrix} \bigr)
= \bar{\theta}( \tau_1' \bar{\tau}_2 - \tau_2' \bar{\tau}_1).
\end{equation*}
(b) To show that $g_1$ is an isometry we calculate as follows:
\begin{equation*}
\ip{g_1 z_1}{ g_1 z_2 } 
= z_1^* \beta^* g^* (\beta^{-1})^* J_F \beta^{-1} g \beta z_2
= \bar{\theta} z_1^* \beta^* g^* S g \beta z_2  
= \bar{\theta} z_1^* \beta^* S \beta z_2  = z_1^* J_F z_2 = \ip{z_1}{ z_2}.
\end{equation*}
The second equality uses \eqref{jbeta} and the third one follows from
$g^{tr} S g = S$ for all $g \in \op{Sp}_2(\ZZ)$.
\par
Let $p_1 = \op{det}(\beta) = 3 - \omega$. Then $p_1 \bar{p}_1 = 13$.
Let $ a - d = s$ and $b + c = t$.
We find that
\begin{equation*}
g_1 = \frac{1}{p_1} \begin{pmatrix} 
 p_1 a + \omega s + \bar{p} t  & -\omega p_1 b -\bar{\omega} p s -\bar{\omega} t \\
 -\bar{\omega} p_1 c - p s - t  & p_1 d - \omega s - \bar{p} t  
\end{pmatrix}.
\end{equation*}
Note that $\omega s + \bar{p} t \equiv \bar{p}(ps + t) \bmod p_1$.
So the entries of the matrix $g_1$ belong to $\cE$
if and only if $ p_1$ divides $(ps + t)$ in $\cE$.
This is equivalent to the congruence condition $ 5 s + t \equiv 0 \bmod 13 \ZZ$. 
On the other hand one checks by direct congruence calculation that
\begin{equation*} 
\nu^{-1} \Gamma_0(13) \nu 
= \Big\lbrace \begin{pmatrix} a & b \\ c & d \end{pmatrix} \in \op{SL}_2(\ZZ) \colon 
5(a - d) + (b + c) \equiv 0 \bmod 13   \Big\rbrace. 
\end{equation*}
\end{proof}
\begin{proof}[proof of theorem \ref{FisoH}] 
The isometry $\beta$ has been defined in \ref{npm}.
The claims $\beta(\br) = i$, $\beta(w_{\cP}) = -\beta(w_{\cL})^{-1} = p$ and
$\beta(w_{\cP} + \theta w_{\cL}) = 0$ are easily checked using \eqref{defbeta}. 
The automorphism $\sigma$ interchanges $w_{\cP} = \beta^{-1}(p)$ and $w_{\cL}= \beta^{-1}(-p^{-1})$
and fixes $\bar{\rho} = \beta^{-1}(i)$. So $c_{\beta}(\sigma) = S$.
\par
Suppose $g \in \op{SL}_2(\ZZ)$ such that $ \nu g \nu^{-1} \in \Gamma_0(13)$. Lemma
\ref{betalemma}(b) shows that $g_1 = \beta^{-1} g \beta$ is an isometry of the lattice
$F$. Since $i_F(\bar{L})$ is generated by $\bar{x}$ (see \ref{FFp} (d)), there exists
$\lambda \in (\cE/13 \cE)^*$ such that $g_1(\bar{x}) \equiv \lambda \bar{x} \bmod F$. 
From \ref{FFp}(e) we find that $g_1$ can be extended to an automorphism of $L$ if
there exists an automorphism of $F^{\bot}$ that acts on $i_{F^{\bot}}(\bar{L})$ as
multiplication by $\lambda$. In particular, \ref{FFp}(f) implies that $g_1$ can be
extended to $L$ if $\lambda$ belongs to the multiplicative subgroup of $(\cE/13 \cE)^*$
generated by $3\bar{p}$. We check this by a direct calculation, sketched below.
The congruences somehow work out exactly as we need.
Looking at this arithmetic it seems that a computation free, conceptual
argument must exist.
\par
Recall from \ref{FFp}(b) that $\pi_F(x) = \Sigma_{\cP}/13$, where
$\Sigma_{\cP} = 4 w_{\cP} - \bar{p} w_{\cL} \in F$. So in our matrix notation,
$ \Sigma_{\cP} =  \SV{4}{-\bar{p}}$. Observe that
$\beta \SV{4}{-\bar{p}} = \det(\beta) \SV{p}{-\bar{\omega}}$. Lemma \ref{betalemma}
implies that 
\begin{equation*}
g \equiv \PM{ s + d }{b}{-5s - b}{d} \bmod 13,
\end{equation*} 
for some $s \in \ZZ$. Using this, one checks that
\begin{equation*}
g_1 \PV{4}{-\bar{p}} 
= \beta^{-1}  g \beta \PV{4}{-\bar{p}} 
\equiv \lambda \PV{4}{-\bar{p}} \bmod 13,
\text{\; where \;}
\lambda  = [8s (4 + 3 \omega) + d + 3p \bar{\omega} b] \bmod 13 \cE.
\end{equation*}
In other words  $g_1(\bar{x}) \equiv \lambda \bar{x} \bmod F$.
Recall the prime factorization $13 = p_1 \bar{p}_1$, where $p_1 = 3 - \omega$.
By Chinese remainder theorem, we have an isomorphism 
\begin{equation*}
\varphi: \cE/13 \cE \to \cE/p_1 \cE \times \cE / \bar{p}_1 \cE \simeq \FF_{13} \times \FF_{13},
\text{\; given by, \;}  x \mapsto ( x \bmod p_1, x \bmod \bar{p}_1 ).
\end{equation*}  
Now, using $\varphi(\omega) = (3,9)$ and $\varphi(\bar{\omega}) = (9,3)$,
it follows that
\begin{equation*}
\varphi(3\bar{p}) = (7, 2) = ( 2^{-1}, 2) \in \FF_{13} \times \FF_{13}
\text{\; and \;}
\varphi( \lambda) = (d + 5 b, s + d + 8 b). 
\end{equation*}
Since $2$ is a generator of $\FF_{13}^*$,
it follows that $\lambda$ belong to the multiplicative group generated by $3\bar{p}$ if
and only if $\varphi( \lambda)$ has the form $(u, v)$ with $u.v = 1$.
To finish the proof, we note that
\begin{equation*}
(d + 5 b)(s + d + 8 b)  \equiv s(d + 5 b) + d^2 + b^2 \equiv \op{det}(g) \equiv 1 \bmod 13. 
\end{equation*}
\end{proof}

\begin{lemma}
(a) Let $z \in F$ be a norm $3$ vector. Then there exists $g_z \in \op{SL}_2(\ZZ)$ such that
$\beta(z) = g_z(\omega)$.  
\newline
(b) Let $r \in F$ be a vector of norm $-3$. Let $z \in F \cap r^{\bot}$ be a primitive non-zero vector.
Then there exists $g_z \in \op{SL}_2(\ZZ)$ such that $\beta(z) = g_z(\omega)$.
\label{omegaconj}
\end{lemma}
\begin{proof}
(a) Let $z \in F$.
Let us write $ \beta.z = \bigl( \begin{smallmatrix} s  \\  t  \end{smallmatrix} \bigr)
=\bigl( \begin{smallmatrix} s_1 \omega + s_2  \\  t_1 \omega + t_2  \end{smallmatrix} \bigr)$
with $s_j, t_j \in \ZZ$.
Part (a) of lemma \ref{betalemma} implies that $\abs{z}^2 = 3$ if and only if 
$\sqrt{3} =  2 \op{Im}( \bar{t} s)$, which simplifies to  $s_1 t_2 - s_2 t_1 = 1$. 
So we can take $g_z = \bigl( \begin{smallmatrix} s_1 & s_2 \\ t_1 & t_2 \end{smallmatrix} \bigr)$.
\newline
(b) Let $\beta.r =  \bigl( \begin{smallmatrix} s  \\  t  \end{smallmatrix} \bigr)$.
Now  $\abs{r}^2 = -3$ implies $2 \op{Im}( \bar{t} s ) = -\sqrt{3}$.
So $2 \op{Im} ( \bar{s} \bar{\bar{t}} ) = \sqrt{3}$. As in part (a), this implies 
$\bar{s}/ \bar{t}  \in \op{PSL}_2(\ZZ)\omega$.
Now let $\beta.z  = \bigl( \begin{smallmatrix} u  \\  v  \end{smallmatrix} \bigr)$.
One has $0 = \ip{z}{r} = \bar{\theta}(-\bar{s} v + \bar{t} u)$. It follows that
$\beta(z) = u/v = \bar{s}/\bar{t} \in \op{PSL}_2(\ZZ)\omega$. 
\end{proof}
\begin{remark}
The stabilizer of $\omega \in \cH^2$ in $\op{PSL}_2(\ZZ)$ is the $\ZZ/3\ZZ$ generated by $ST$.
If $z \in F$ has norm 3, then the
$120^{\circ}$ rotation around $z \in \PP_+(F)$ belongs to $\Gamma_F$. If $r \in F$ is a vector of norm $-3$
and $z$ is a primitive vector in $F \cap r^{\bot}$, then the restriction of $\phi_r^{\omega}$ to $\PP_+(F)$
is an element of $\Gamma_F$, which is again a $120^{\circ}$ rotation around $z$.
Lemma \ref{omegaconj} implies that the image of these rotations 
under $c_{\beta}$ are $\op{PSL}_2(\ZZ)$ conjugates of $ST$ or $(ST)^{-1}$.
\end{remark} 
%
%*******************************************************************************************************
%
%
%*******************************************************************************************************
%
%
%
%*******************************************************************************************************
%
%

%
\end{document}